\documentclass[arxiv,nonblindrev]{arXivC}

\DoubleSpacedXI 

\usepackage{amsmath,amssymb,tikz,pgfplots,mathdots,dutchcal,subfig}
\usetikzlibrary{decorations.pathreplacing}
\usetikzlibrary{automata, positioning,shapes.geometric}
\usetikzlibrary{chains,shapes.multipart}

\newtheorem{Proposition}{Proposition}
\newtheorem{Corollary}{Corollary}

\usepackage{natbib}
 \bibpunct[, ]{(}{)}{,}{a}{}{,}%
 \def\bibfont{\footnotesize}%

\DeclareMathOperator{\expd}{Exp}
\DeclareMathOperator{\erld}{Erlang}

\newcommand{\rnot}{R_0^{\mathsf{sys}}}
\newcommand{\tsf}{\mathsf{T}}
\newcommand{\hsf}{\mathsf{H}}
\newcommand{\lsf}{\mathsf{L}}
\newcommand{\abf}{\mathbf{A}}
\newcommand{\bbf}{\mathbf{B}}
\newcommand{\lhr}{\lambda_\hsf}
\newcommand{\llr}{\lambda_\lsf}
\newcommand{\rnothr}{R_0^{H}}
\newcommand{\rnotlr}{R_0^{L}}

\newcommand{\rhhb}{R_0^{\hsf\overset{\bbf}{\to}\hsf}}

\newcommand{\rhlb}{R_0^{\hsf\overset{\bbf}{\to}\lsf}}

\newcommand{\rllb}{R_0^{\lsf\overset{\bbf}{\to}\lsf}}
\newcommand{\rlla}{R_0^{\lsf\overset{\abf}{\to}\lsf}}
\newcommand{\rlhb}{R_0^{\lsf\overset{\bbf}{\to}\hsf}}
\newcommand{\rlha}{R_0^{\lsf\overset{\abf}{\to}\hsf}}
\newcommand{\rta}{R_0^{\tsfo\overset{\abf}{\to}\tsft}}
\newcommand{\rtb}{R_0^{\tsfo\overset{\bbf}{\to}\tsft}}

\newcommand{\qh}{q_\hsf}
\newcommand{\ql}{q_\lsf}
\newcommand{\qt}{q_\tsf}
\newcommand{\tsfo}{\tsf_{1}}
\newcommand{\tsft}{\tsf_{2}}

\newcommand{\psys}{\pi}

\newcommand{\wll}{W_{\lsf\to\lsf(i)}^{(h,\ell)}}

\RUNAUTHOR{Kang, Doroudi, Delasay, and Wickeham}

\RUNTITLE{Transmission Risks In Service Facilities}

\TITLE{\Large{A Queueing-Theoretic Framework for Evaluating Transmission Risks in Service Facilities During a Pandemic}}

\begin{document}

\ARTICLEAUTHORS{
\AUTHOR{Kang Kang\thanks{Department of Industrial and Systems Engineering, University of Minnesota, Minneapolis, MN 55455}}
\AFF{\EMAIL{kangx747@umn.edu}}
\AUTHOR{Sherwin Doroudi\footnotemark[1]}
\AFF{\EMAIL{sdoroudi@umn.edu}}
\AUTHOR{Mohammad Delasay\thanks{College of Business, Stony Brook University, Stony Brook, NY 11794}}
\AFF{\EMAIL{mohammad.delasay@stonybrook.edu}}
\AUTHOR{Alexander Wickeham\footnotemark[1]}
\AFF{\EMAIL{wicke086@umn.edu}}
} 

\ABSTRACT{
We propose a new modeling framework for evaluating the risk of disease transmission during a pandemic in small-scale settings driven by stochasticity in the arrival and service processes, i.e., congestion-prone confined-space service facilities, such as grocery stores. We propose a novel metric ($\rnot$) inspired by $R_0$, the ``basic reproduction number'' concept from epidemiology, which measures the transmissibility of infectious diseases. We derive the~$\rnot$ metric for various queueing models of service facilities by leveraging a novel queueing-theoretic notion: sojourn time overlaps.  We showcase how our metric can be used to explore the efficacy of a variety of interventions aimed at curbing the spread of disease inside service facilities. Specifically, we focus on some prevalent interventions employed during the COVID-19 pandemic: limiting the occupancy of service facilities, protecting high-risk customers (via prioritization or designated time windows), and increasing the service speed (or limiting patronage duration).  We discuss a variety of directions for adapting our transmission model to incorporate some more nuanced features of disease transmission, including heterogeneity in the population immunity level, varying levels of mask usage, and spatial considerations in disease transmission.
}

\KEYWORDS{COVID-19 pandemic, Disease spread, Service systems, Queueing theory, Epidemiology}

\maketitle

\section{Introduction}\label{sec:intro}
Research on the spread of COVID-19---and efforts to curb it---have so far focused primarily on two categories of processes: (i) the biological and physical processes that govern the spread of viral particles (e.g., through fomites, respiratory droplets, and aerosolized particles) and (ii) the positioning and movement of humans throughout and across their communities, as they fulfill their various wants and needs (e.g., to work, shop for groceries, seek healthcare, etc.). Key efforts to mitigate the spread of disease have focused on both fronts.  Attempts to directly interfere with the spread of viral particles when people are in close proximity to one another include the installation of transparent barriers, use of air filters, and wearing of masks.  Meanwhile, \emph{social/physical distancing} interventions attempt to reduce the extent to which people are in close proximity to one another altogether. Examples of such interventions in service facilities, including retail stores, are  managing traffic flow through one-way movements in the isles, restricting the number of customers (for example, one in, one out admission policies), and facilitating a minimum distance of six feet between customers (for example, by using floor stickers)~\citep{bove2020restrict}.

To date, little analytic work has emerged that explores the interplay between these two categories of processes---(i) and (ii) mentioned above---in small-scale (as opposed to community- or population-level) settings.  We fill this gap by interweaving \emph{quantitative descriptions} of these two categories of processes into stochastic models that aid decision makers in the assessment of transmission risks in congestion-prone service facilities under a variety of interventions. For example, consider what a simple quantitative description of (i) the process of viral transmission might look like in a grocery store. A function could describe the likelihood that an infectious customer will transmit enough aerosolized viral particles to a susceptible customer so as to eventually infect the latter, given that their sojourns in the store overlapped for some duration of time.  Meanwhile, a quantitative description of (ii) human movement might consist of the \emph{stochastic} pattern by which both infectious and susceptible customers arrive to the store over time and how long each customer spends in the store.  Our work shows how one can deduce relevant transmission risk measures (e.g., the average number of transmissions in this store per month) from these two quantitative descriptions.

In developing our models, we must take into account the fact that human movement patterns exhibit \emph{idiosyncratic} or \emph{stochastic} variation.  The grocery store described above may be particularly busy at 10 A.M. on a given Tuesday---not because that time of day is typically popular for shopping on Tuesdays (\emph{systematic} variation)---but because it just so happened by chance that many shoppers decided to visit the store around 10 A.M. on that particular Tuesday (\emph{stochastic} variation). Such variation suggests that the store runs long periods of time without any transmissions, but then occasionally exhibits significant transmission events where a single infectious customer (or staff member) infects multiple susceptible individuals.  Such an infection pattern is consistent with observations that SARS-CoV-2 transmission exhibits a significant degree of \emph{overdispersion}: most of those who are infected in turn infect very few (if any others), but a small minority of the infected are responsible for the bulk of all transmission, often infecting a number of others in the same environment at roughly the same time \citep{Adam2020, althouse2020stochasticity}. Meanwhile, the fact that some service systems have transitioned from accepting walk-in customers to running by appointment~\citep{BURSTEIN2020} suggests an awareness that stochasticity can drive infection.  Nevertheless, very little analytic work has been devoted to capturing these effects. Our work is a step toward filling this lacuna.

For many decades, the mathematical discipline of queueing theory has in large part been concerned with studying the impact of stochastic variation in both arrival and service patterns on \emph{waiting time}.  Our work illustrates how queueing-theoretic notions can also be adapted to study the effect of stochasticity on each customer pair's \emph{sojourn time overlap}---the duration of time that a pair of customers are both present in the system.  We then leverage our analysis of sojourn time overlaps to obtain expressions for~$\rnot$---a new measure of risk that we propose as the service-center specific analogue of the basic reproduction number~$R_0$ taken from the epidemiological literature---under a variety of classical queueing systems.

Throughout this paper we take classic queueing models (i.e., the~M/M/1,~M/M/$c$, and~M/M/$c$/$k$ queues) that will serve as \emph{exemplars}, allowing us to demonstrate the computation of our novel~$\rnot$ metric in relatively simple settings.  These models highlight the ways in which our approach can evaluate the efficacy of a variety of interventions, including limiting the occupancy of service facilitates and prioritizing the service of high-risk customers. Therefore, the primary contribution of this paper is the introduction of and elaboration on a \emph{modeling framework} consisting of (i) the introduction of new metrics (e.g., $\rnot$), (ii) a flexible set of assumptions regarding disease transmission (i.e., a disease transmission model---with additional variations discussed in the concluding section of this paper), and (iii) worked examples demonstrating how novel queueing-theoretic notions (e.g., sojourn time overlaps) can be used to obtain exact results for these metrics across a variety of systems and interventions.  While the nuances of comparatively few real-world service systems are best captured by models as simple as the M/M/1 and M/M/$c$ queues, we hope our exemplars will serve as blueprints allowing for future research to determine the novel~$\rnot$ metric (or other metrics inspired by~$\rnot$) for more complicated queueing (or other stochastic) models that best capture the features called for by a given problem of interest.

A brief summary of the specific contributions and examples explored in the paper are as follows: we introduce our primary transmission model and the novel $\rnot$ metric, we obtain exact results for $\rnot$ in both the M/M/1 (Section~\ref{sec:mm1}) and M/M/$c$ (Section~\ref{sec:mmc}) queueing systems. We then explore the $\rnot$ metric under three families of interventions: limiting service facility occupancy (Section~\ref{sec:limitedOccupancy}), protecting high risk customers through prioritization (Section~\ref{sec:priorities}) or dedicated time windows (Section~\ref{sec:DesignatingShoppingTimes}), and increasing the service rate of the system (Section~\ref{sec:serviceRate}).  Finally, we show how our modeling framework can flexibly accommodate a variety of additional features: heterogeneity in customer infectiousness and susceptibility (Section~\ref{sec:extension-het}), the possibility of the simultaneous presence of multiple infectious customers within the service facility (Section~\ref{sec:extension-groups}), and the potential impact of distance on disease transmission (Section~\ref{sec:extension-distance}).  

\section{Literature Review}\label{sec:lit}
Most of the research modeling pandemics prior to COVID-19 focuses on disease spread and control at the \emph{population level}, including those that use modifications of classic compartment models which simplify the mathematical modeling of infectious diseases by assigning the population to compartments (for example, Susceptible, Infectious, and Recovered in the SIR model~\citep{weiss2013sir}). Such models have been used for the spread of the COVID-19 pandemic and evaluation of the efficiency of various population-level interventions, including social distancing guidelines~\citep{ housni2020future,Kaplan2020}, cost-benefit analysis of lockdown policies~\citep{acemoglu2020multi, alvarez2020simple, glover2020health}, and targeted restrictions on ``superspreader'' locations~\citep{Chang2021}. Compartment models have also been used in conjunction with spatial epidemic spread models to incorporate the movement of people from one location in a community to another~\citep{Balcan2009,drakopoulos2017eliminate, drakopoulos2017network}. For example, \cite{birge2020controlling} analyze a spatial epidemic spread model suggesting that targeted closures curb the spread of the COVID-19 epidemic at substantially lower economic losses than city-wide closure policies. Other recent research in this stream include~\cite{Chinazzi2020,Jia2020}, and~\cite{Chang2021}. As opposed to these models and their focus on population- and community-level disease spread, our focus is on the spread of disease during a pandemic at small-scale settings like a service facility that is prone to congestion that is driven by idiosyncratic stochasticity.  

Queueing models are designed to capture stochasticity in the operations of service facilities and customers' sojourn in such systems. Most of the queueing work related to epidemics also focuses on population-level disease spread and control. For example, \cite{kumar1981some}, \cite{trapman2008useful}, \cite{dike2016queueing}, and \cite{singh2018markovian} borrow standard queueing-theoretic notions such as M/M/1,
M/G/1, and busy-period analysis in order to investigate the efficiency of intervention strategies pertaining
to quarantine centers, vaccination, etc. The COVID-19 pandemic has spurred renewed interest in this topic with researchers mainly focusing on evaluating the effectiveness of various controls~\citep{alban2020icu,cui2020design,long2020pooling,meares2020system,palomo2020flattening}. The use of queueing theory in this paper differs from that in the works cited above in  that our models are designed to capture the spread of disease---and evaluating the effectiveness of a variety of interventions---in \emph{small-scale settings} (i.e., service facilities) by capturing the stochasticity of the disease transmission as well as stochasticity of customers' sojourn.

Relatively little of the large body of quantitative research on the COVID-19 pandemic that has developed in the past year has focused on the spread of disease---and its mitigation---in small-scale settings.  We briefly discuss the work devoted to this topic.  First, \cite{shumsky2020retail} model customers' movement in a single-aisle grocery store and show that one-way movement lowers the transmission risk significantly if transmission occurs only when two customers are in close proximity. Their transmission model uses customers' proximity to estimate exposure risks. While proximity certainly plays a role in the transmission of SARS-CoV-2, our proposed model provides a complementary perspective on transmission by emphasizing duration of exposure rather than distance (which we also address in Section~\ref{sec:extension-distance}).\footnote{Our emphasis on duration of exposure rather than distance is motivated by e.g., \cite{Zhang14857} that caution against the use of models driven purely by proximity as they may be insufficient to address aerosol transmission, \cite{Stadnytskyi2020} that present evidence for  the long lifetime of aerosol particles, and \cite{zhang2015} that discuss rapid air-mixing.}  Second, \cite{garcia2020assessment} develop models to assess the COVID-19 transmission risk in  outdoor  crowds by capturing details such as people's  physical  distance  and head orientation. They find that street caf\'{e}s and venues where people form queues present a large average rate of new infections caused by a customer when customers' close proximity was prolonged over considerable time.

Third, \cite{Tupper32038} extend the epidemiological basic reproduction number $R_0$ and introduce an analogous measure, ``event $R$,'' which represents the expected number of new infections which occur at an event as a result of contact with a single infected individual. This measure captures four transmission factors: intensity, duration of exposure, the proximity of individuals, and the degree of mixing (people interacting with several groups).  Our models shares many features with those of this paper (e.g., we also develop a related metric, $\rnot$), although crucially unlike our work, they do not consider arrival or departure processes (i.e., queuing dynamics).\footnote{Note that \cite{Tupper32038} are studying disease transmission during \emph{events}, and so they understandably need not be concerned with queueing dynamics.  Meanwhile,  as we are studying disease transmission in \emph{service facilities} (facing arrivals subject to idiosyncratic stochastic variation), our work essentially necessitates the consideration of such dynamics.}   In fact, none of the three papers discussed above incorporate the stochasticity involved in the congestion of these small-scale systems, which affects the duration of exposure and transmission intensity, and consequently, the risk of infection.

The paper that is closest to ours with respect to incorporating the aforementioned stochasticity is the work by~\cite{perlman2020reducing} who use an existing queueing-theoretic metric (specifically, the second factorial moment of the number of customers in the service facility) as a proxy for the level of transmission risk in the system. They provide a method for calculating this metric in a detailed model of a grocery store retailer. We develop a detailed metric of viral transmissibility in a queueing system by taking into account the distribution of the duration of sojourn time overlaps between pairs of customers.

We provide a detailed comparison between our proposed metric and the one proposed in~\cite{perlman2020reducing} in Section~\ref{sec:yechiali}. While it is beyond the scope of this paper, our metric can be computed for more complicated systems than those studied here, including the aforementioned grocery store model presented in~\cite{perlman2020reducing}.  In yet more recent work~\cite{perlman2021impact} consider a different queueing-theoretic measure of risk: specifically, each customer faces a risk proportional to the number of customers they ``meet'' while waiting outside of the service facility.\footnote{This metric results in an aggregate rate of risk equal to the arrival rate multiplied by the time average number of customers waiting outside the service facility.}  Unlike our work, they do not consider the duration of such ``meetings.''

\section{Modeling Disease Transmission in a Service Facility}
\label{sec:model}
Our goal is to assess the risk of disease transmission in service facilities through our proposed~$R_0^{\text{sys}}$ metric, which is a reinterpretation of the epidemiological concept  of the \emph{basic reproductive rate}. In this section, we first discuss the dynamics of disease transmission and customers' movement inside a service facility and, later in Subsection~\ref{sec:rnot}, we elaborate the~$R_0^{\text{sys}}$ metric.

We consider service facility settings in which customers arrive, spend some time in the facility receiving service, and then leave (for example, a retail store or a bank branch). Customers entering the service facility are either \emph{infectious} (capable of infecting others) or \emph{susceptible} (capable of being infected by others). Of course, customers may not fall neatly into these categories (e.g., they are infected but not yet infectious---referred to as ``exposed'' in the epidemiological literature \citep{brauer2019mathematical}---or they might have immunity derived from prior infection or vaccination). We address the cases of immunity in Section~\ref{sec:extensions}.

Our viral transmission model inside the service facility derives from the exponential dose-response model, frequently used in the study of viral transmission, as follows: A susceptible customer that coexists in the service facility with an infectious customer is exposed to viral particles at a constant rate across time, and each exposure has a (very small) probability of causing the susceptible customer to become infected. Treating these potential infection events as independent, the number of exposures required for infection is geometrically distributed, and taking these exposures to be happening very frequently across time, the \emph{transmission threshold}~$\theta$---i.e., the amount of time the sojourns of the infectious and susceptible customers must overlap before the susceptible customer becomes infected---follows an exponential random variable with some rate~$\alpha$ (i.e., the mean transmission threshold is $1/\alpha$); see~\cite{Szeto2010} for more information on this model.

Estimating the transmission rate~$\alpha$ empirically can prove difficult; see~\cite{Watanabe2010} and~\cite{Zhang2020} for examples of work that attempt to estimate the transmission threshed~$\theta$ for the SARS-CoV-1 (i.e., the agent that led to the 2003 SARS outbreak) and SARS-CoV-2 viruses (i.e., the agent that led to the COVID-19 pandemic), respectively.  Of course, the transmission threshold~$\theta$ may depend on the particular pair of infectious and susceptible customers (e.g., due to the type of mask they are wearing if any, how they are breathing, etc.); we discuss such possibilities in further detail in Section~\ref{sec:extension-het}.  More generally, some models also suggest that the transmission threshold~$\theta$ follows a non-exponential distribution. For example, the beta-Poisson dose response model is sometimes more plausible when a population exhibits highly heterogeneous susceptibility to the pathogen~\citep{Moran1954}; however, fitting the parameters of such models requires substantially detailed data, which is often not available.  Various other phase-type distributions have proved useful in modeling dose-response relationships and should be considered when, for example, arrivals are highly variable (super-Poisson)~\citep{kuhl2006receptor, kuhl2007nonexponential}.  

In section~\ref{sec:extension-het} we show how our analysis can be modified to accommodate a hyper-exponential dose-response relationship.

Fig.~\ref{fig:example} shows an example of a sample path of the arrival and service processes of four customers in a service facility where customer~3 is infectious while others are susceptible. The sojourn of customers~1,~2, and~4 overlap with that of the infectious customer~3 for~10,~30, and~20 minutes, respectively. The plot on the right specifies the probability that each of the susceptible customers is infected during their sojourn time and due to their overlap with the infectious customer, if the transmission threshold~$\theta$ is exponentially distributed with mean of~$15$ minutes. On high level, these probabilities and the corresponding overlap times form the basis of the computation of our~$\rnot$ metric, though aggregated over all possible sample paths. 
\begin{figure}[]
\centering
\includegraphics[width=\textwidth]{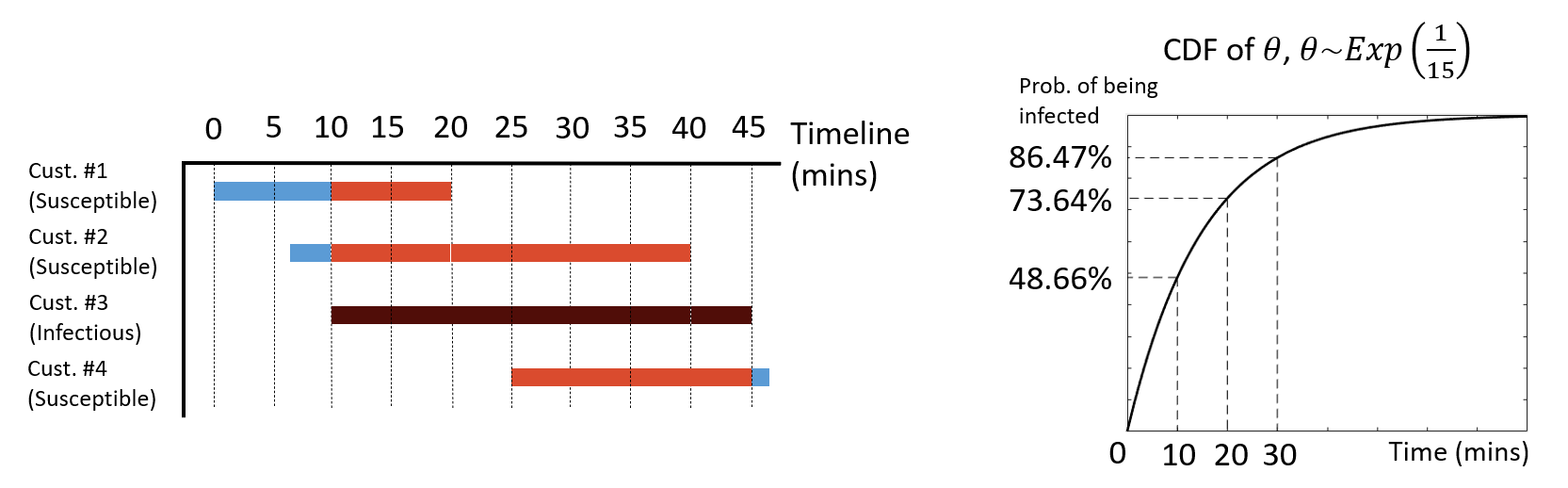}
\caption{An illustration of susceptible customers' sojourns and sojourn time overlaps with an infectious customer (customer 3). The right plot shows the CDF of~$\theta\sim \expd(1/15)$ and specifies the infection risk of the susceptible customers.}
\label{fig:example}
\end{figure}
  
\subsection{$\rnot$ and related metrics measuring public health risks}
\label{sec:rnot}
Our novel metric~$\rnot$ captures the stochastic dynamics of service facilities to measures public health risk during a pandemic, which is a reinterpretation of the epidemiological concept of the the \emph{basic reproductive rate}, $R_0$, which measures the transmissiblity of infections diseases~\citep{heffernan2005perspectives, liu2020reproductive}.  The $R_0$ value associated with an epidemic is the expected number of transmissions from an infected individual in a population where all other individuals are susceptible. Our proposed~$\rnot$ metric is an analogous metric that measures \emph{the expected number of transmissions from an infected customer during a single sojourn in a service facility, assuming all other customers in the facility are susceptible}.
\footnote{For simplicity, our transmission model (and hence, our definition of $\rnot$) disregard the possibility of staff members infecting (or becoming infected by) customers and/or one another.  While our model can be modified in a straightforward fashion to capture infections between staff and customers in a ``naive'' sense, modeling the number of infectious staff members at any given time will present its own challenges.  In practice, whenever a staff member is infectious, it is likely that one or more of their colleagues are also (or will soon become) infectious, since staff spend a lot of time with each other in the service facility.  Additionally, staff who are suspected of be infected are often encouraged to stay home from work; thus, it may be prudent to include a modeling feature that allows for the occasional removal of staff members.  These considerations are beyond our scope, but may provide fruitful avenues for future work.}  Hence, while the classical $R_0$ value associated with an epidemic captures transmissibility in all contexts within a community, $\rnot$ only captures transmissibility within a single service facility, and only during a single visit to said facility. Note that just as $R_0$ depends on a variety of community-specific factors beyond the biological and physical properties of the contagion (e.g., population density, traveling patterns, government interventions and recommendations, compliance with said interventions and recommendations, etc.), $\rnot$ will depend on many of the attributes specific to the service facility in question (e.g., arrival rates, service rates, and system design interventions).\footnote{We note that one important difference between the original $R_0$ metric and $\rnot$ is that the former measures transmission risk in a closed-loop system (i.e., a high $R_0$ value leads to what is initially an exponential growth in the spread of the disease); $\rnot$ exists within an open-loop system, i.e., we treat the service facility as a small part of the community, where an increase in infection at the facility is unlikely to cause a significant feedback in terms of an increase in infectious arrivals at the facility; i.e., we treat the probability of a customer being infectious as exogenous.}

The $\rnot$ metric allows for the derivation of other metrics of interest.  For example, if we assume that customers arrive with rate~$\lambda$ and are infectious with some small probability~$p$, then the rate of new infections at a given service facility can be approximated by~$\lambda p\rnot$, so long as~$p$ is small.  Similarly, the mean fraction of customers that are infected can be approximated by~$p\rnot$.   
   
\subsection{Model Implications:  Comparing the $\rnot$ and $\mathbb E[N(N-1)]$ metrics}
\label{sec:yechiali}
\cite{perlman2020reducing} propose~$\mathbb E[N(N-1)]$ (where~$N$ is the number of customers in the system) as a measure of public health risk in a queueing model of a grocery store.  This metric gives the time average number of \emph{pairs} of customers in the system.  Each pair represents a possible transmission, as one customer in the pair may be infectious and the other susceptible.  While this metric represents a reasonable candidate for studying public health risks due to congestion in queueing systems (especially, as existing methods exist for its computation across a variety of system settings), our~$\rnot$ metric captures certain features of disease transmission (and therefore distinguish between varying levels of risk) in ways that the~$\mathbb E[N(N-1)]$ metric does not.  In particular, there exist systems with identical~$\mathbb E[N(N-1)]$ values but considerably different~$\rnot$ values, which makes~$\mathbb E[N(N-1)]$ unable to assess the efficacy of some interventions.

The following example, although artificial, illustrates the primary shortcoming of the~$\mathbb E[N(N-1)]$ metric, which is that it does not account for transmission times and thresholds. Consider a~$D^m/D^m/1$ system with deterministic arrival and service processes where~$m$ customers arrive simultaneously every~$1/\lambda$ time units (i.e., exactly at times~$t=0,1/\lambda,2/\lambda,\ldots$), and a server serves the~$m$ customers as a batch after they have been in the system for~$1/\mu$ time units (where $\mu>\lambda$ in order to ensure system stability), so that the number of customers in the system at time $t$ is given by $N(t)=m$ during time intervals $(0,1/\mu)$, $(1/\lambda,1/\lambda+1/\mu)$, $(2/\lambda,2/\lambda+1/\mu)$, etc., and~$N(t)=0$ during the time intervals~$(1/\mu,1/\lambda)$, $(1/\lambda+1/\mu,2\lambda)$, $(2/\lambda+1/\mu+3/\lambda)$, etc. It is straightforward to obtain~$\mathbb E[N(N-1)]=\lambda m(m-1)/\mu$ for this system.

Meanwhile, we can compute~$\rnot$ as follows: if we assume that one arrival is infectious, while others are susceptible, then the infectious arrival's sojourn will coincide with the~$m-1$ other customers who arrived (and will depart) with that customer. Their sojourn will not overlap with that of any other customers.  Hence, by the linearity of expectation, the infectious customer will infect $\rnot=(m-1)\mathbb P(\theta\ge 1/\mu)$ other customers on average. Now assume that each arrival is infectious, independent of all others, with some probability $p\ll1/m$, allowing us to safely ignore the possibility of there ever being more than one infectious arrival in the system.  Then, since this system features an infectious customer arrival rate of $pm\lambda$, the rate at which customers become infected (under our transmission assumptions) is \[pm\lambda\rnot=p\lambda m(m-1)\mathbb P(\theta\ge 1/\mu)=p\mu\mathbb P(\theta\ge1/\mu)\mathbb E[N(N-1)].\]  Hence, under our transmission assumptions, within this context, $\mathbb E[N(N-1)]$ is an appropriate choice for evaluating the efficacy of various interventions if and only if $p\mu\mathbb P(\theta\ge1/\mu)$ remains (at least nearly) unchanged as a result of these interventions.   If we further assume that $\theta\sim\expd(\alpha)$, the aforementioned quantity becomes $p\mu\left(1-e^{-\alpha/\mu}\right)$.  Presumably, most \emph{operational} interventions will leave $p$ and $\alpha$ unchanged, and moreover $p\mu(1-e^{-\alpha/\mu})\approx p\alpha$ when $\alpha\ll\mu$. So, in a regime where transmission takes a very long time on average relative to the service rate,~$\mathbb E[N(N-1)]$ can even be used to assess interventions that change~$\mu$ (so long as it is kept in this regime). But, when~$\alpha/\mu$ is not negligibly small, then interventions that change~$\mu$ may not be adequately assessed by the~$\mathbb E[N(N-1)]$ metric. For example, an intervention which leads to both doubling~$\lambda$ and~$\mu$, which leaves~$\mathbb E[N(N-1)]$ unchanged, can actually have a significant effect on reducing transmissions, as individual customers cut their exposure times in half.

In fact under the exponential dose response model (i.e., when $\theta\sim\expd(\alpha)$), what the $\mathbb E[N(N-1)]$ metric captures is a measure of risk that measures the number of times infection events would have occurred under the assumption that a customer who has already become infected can become infected again (during the same sojourn in the service facility).  This type of framework would suggest that spending 100 time units with one infected individual is 100 times worse (in terms of some expected healthcare risk) than spending 1 unit of time with the same individual.  Note however that the likelihood of becoming infected more than once in this hypothetical sense is equal to the likelihood of experiencing two arrivals in a Poisson process with rate $\alpha$ in an interval of length~$1/\mu$ which is~$o(\alpha/\mu)$, and hence, negligible when $\alpha\ll\mu$, which explains why an~$\mathbb E[N(N-1)]$-driven analysis agrees with our~$\rnot$-driven analysis in this regime.

While we have focused on an artificial model in our discussion, it is important to note that the differences between the $\mathbb E[N(N-1)]$ and $\rnot$ metrics persist across a variety of settings (although their quantitative relationships can be setting-dependent).  We will discuss one other crucial difference between these metrics later in the paper: due to its abstraction of transmission dynamics, the $\mathbb E[N(N-1)]$ can exhibit insensitivity to scheduling policies in the presence of exponentially distributed service requirements, making it unsuitable for assessing scheduling-based interventions (see Section~\ref{sec:priorities}).

\section{Analytic Methodology and Worked Examples}\label{sec:analysis}
We find~$\rnot$ using a mix of transient and steady-state queueing analysis that tracks what happens during the sojourn of an infectious customer under the assumption that all other customers in the system during this sojourn are susceptible, and once a susceptible customer becomes infected they do not become infectious during their sojourn. Throughout this paper, we assume that customers arrive to an ergodic queueing system according to a Poisson process.  We further make the natural assumption that infectious customers are \emph{functionally indistinguishable} from their susceptible counterparts (e.g., their services times and the rules by which they are scheduled do not depend on their infectious/susceptible status).\footnote{Formally, if we index successive arrivals by natural numbers, then this assumption means that the random sequence of arrival-departure time pairs $\{(A_i,D_i)\}_{i\in\mathbb{N}}$ is independent of the sequence $\{I\{\mbox{arrival $i$ is infectious}\}_{i\in\mathbb{N}}$, where $A_i$ and $D_i$ are the arrival and departure times of customer $i$ and $I\{\cdot\}$ denotes the indicator function.  This assumption holds for arbitrary probabilistic structures on these sequences, so long as they are independent of one another.}

An infectious customer (henceforth, IC) arrives while seeing a system state~$s\in\mathcal S$, where~$\mathcal S$ represents a countable state space.  Let~$n(s)$ be the number of \emph{other} customers in the system when the IC arrives; for simple systems, such as~M/M/1, naturally the state represents the number of customers in the system, i.e.,~$s= n(s)$ for all~$s\in\mathcal S$. Furthermore, let~$\pi(s)$ be the limiting probability that the system is in state~$s$ under steady state.  By the PASTA property \citep[][Chapter 13.3]{harchol2013performance}, with probability~$\pi(s)$ the IC finds the system in state~$s$ and sees~$n(s)$ other customers in the system upon its arrival.  We index the~$n(s)$ customers by~$i\in\{1,2,\ldots,n(s)\}$ according to some convenient indexing scheme (e.g., we could let customer~$i$ be the~$i$-th to arrive among these~$n(s)$ customers).  Now denote by~$W_{i}^{(s)}$ the \emph{sojourn time overlap} between the IC and customer~$i$.  That is,~$W_{i}^{(s)}=\min(d_i,d)-a$, where~$d_i$ is the departure time of customer~$i$, and~$a$ and~$d$ are the arrival and departure times of the IC, respectively. Customer~$i$ becomes infected if and only if~$W_{i}^{(s)}>\theta_i$, where~$\theta_i$ is the random threshold such that the IC infects customer~$i$ if and only if their sojourns overlap for at least this duration of time; transmission thresholds~$\theta_1,\theta_2,\ldots,\theta_{n(s)}$ are assumed to be independent and identically distributed (for a discussion of what can be viewed as a relaxation of this assumption, see Section~\ref{sec:extension-het}). We write~$\theta$ to denote an arbitrary random variable drawn from the same distribution as~$\theta_i$ (for all~$i\in\{1,\ldots,n(s)\}$); we assume that~$\theta$ is independent of all other random variables of interest. The following key result allows for the computation of~$\rnot$:
\begin{Proposition}\label{prop:main}
The expected number of customers that an infectious customer will infect (assuming all other customers are susceptible) is given by
\begin{align}\label{eqn:r0}
\rnot&=2\sum_{s\in\mathcal S}\pi(s)\sum_{i=1}^{n(s)}\mathbb P\left(W_i^{(s)}\ge\theta\right),
\end{align}
where transmission threshold~$\theta$ can be any generally distributed threshold time.
\end{Proposition}
\begin{proof}{Proof of Proposition~\ref{prop:main}}
Given that the IC arrives to a system seeing state~$s$, we find the expected number of customers who will become infected by the IC \emph{among those present in the system upon the IC's arrival}.  In our model, the IC infects customer~$i\in\{1,2,\ldots,n(s)\})$ if and only if~$W_i^{(s)}\ge\theta_i$.  Hence, it follows from the linearity of expectation that the expected number of customers that IC will infect among the~$n(s)$ customers follows:
\[\mathbb E\left[\sum_{i=1}^{n(s)}I\left\{W_i^{(s)}\ge\theta_i\right\}\right]=\sum_{i=1}^{n(s)}\mathbb E\left[I\left\{W_i^{(s)}\ge\theta_i\right\}\right]=\sum_{i=1}^{n(s)}\mathbb P\left(W_i^{(s)}\ge\theta_i\right).\] 
Next, we assume that the system is in steady state (rather than assuming it is in a given state~$s$).  We then condition on the state~$s\in\mathcal S$ observed by the IC upon their arrival, which allows us to deduce that the expected number of (pre-existing) customers that will be infected by the IC is \begin{align}\label{eq:half}\sum_{s\in\mathcal S}\pi(s)\sum_{i=1}^{n(s)}\mathbb P\left(W_i^{(s)}\ge\theta\right).\end{align}
We prove Eq.~\ref{eqn:r0} by arguing that $\rnot$ is precisely equal to twice the quantity given above in Display~\eqref{eq:half} by way of a symmetry argument, where we show that the distribution of the number of customers that the IC infects among those who arrived to the system \emph{before} the IC is equal to (but not necessarily independent of) to the distribution of customers that the IC infects among those who arrived to the system \emph{after} the IC.  For further details, see Appendix~\ref{app:double}.\hfill~$\square$
\end{proof}

Proposition~\ref{prop:main} allows us to obtain~$\rnot$ exactly whenever we can compute the cumulative distribution function (CDF) of~$W_{i}^{(s)}$ and~$\pi(s)$ exactly for all~$s\in\mathcal S$ and~$i\in\{1,2,\ldots,n(s)\}$, although we may not be able to obtain a closed-form expression for~$\rnot$ when the state space~$\mathcal S$ is infinite.  As it turns out, for the exponential dose response model where transmission thresholds are exponentially distributed (i.e.,~$\theta\sim\expd(\alpha)$ for some transmission rate~$\alpha$), we can obtain closed-form expressions for~$\rnot$ for the~M/M/1 and~M/M/c models based on the Laplace transforms of the~$W_i^{(s)}$ random variables.

\begin{Corollary}\label{cor:main}
As long as transmission thresholds~$\theta\sim\expd(\alpha)$, then 
\begin{align}\label{eqn:r0exp}
\rnot&=2\sum_{s\in\mathcal S}\pi(s)\sum_{i=1}^{n(s)}\left(1-\widetilde{W}_i^{(s)}(\alpha_i)\right)=2\left(\mathbb E[N]-\sum_{s\in\mathcal S}\pi(s)\sum_{i=1}^{n(s)}\widetilde{W}_i^{(s)}(\alpha_i)\right),
\end{align}
where~$\widetilde{W}_i^{(s)}$ is the Laplace transform of~$W_i^{(s)}$ and~$\mathbb E[N]$ is the time-average number of customers in the system (i.e.,~$\mathbb E[N]=\sum_{s\in \mathcal S} n(s)\pi(s)$). 
\end{Corollary}
\begin{proof}{Proof of Corollary~\ref{cor:main}}
Eq.~\ref{eqn:r0exp} follows directly from Eq.~\ref{eqn:r0} noting that (i) for any random variable~$X$ for which a Laplace transform~$\widetilde X$ exists and any random variable~$\theta\sim\expd(\alpha)$ independent of~$X$, we have~$\mathbb  P(X<\theta)=\widetilde X(\alpha)$ \citep[see exercise 25.7 of][]{harchol2013performance} and (ii)~$\mathbb E[N]=\sum_{s\in\mathcal S} n(s)\pi(s).$ \hfill~$\square$
\end{proof}

\subsection{Finding~$\rnot$ for the M/M/1 model}
\label{sec:mm1}
We now turn to using Eq.~\ref{eqn:r0exp} from Corollary~\ref{cor:main} in order to derive~$\rnot$ for the classic M/M/1 first-come-first-serve (FCFS) queueing systems where customers arrive with rate~$\lambda$ and are served with rate~$\mu$ and the system load~$\rho\equiv\lambda/\mu$. The M/M/1 system is a special case of the M/M/$c$ system (which we will analyze later), but we examine this special case separately for illustrative purposes. We introduce the \emph{normalized transmission rate}~$\eta\equiv\alpha/\mu$, where~$1/\eta$ corresponds to the average number of service durations that an infectious customer's sojourn will need to overlap with that of a susceptible customer before the former infects the latter. We denote each state by~$s$, the number of customers in the system (i.e.,~$n(s)=s$) and index customers in arrival order (e.g.,~$s=3$ denotes that customer 1 is in service and customers 2 and 3 are waiting); the state space is~$\mathcal S=\{0,1,2,\ldots\}$.
\begin{Proposition}\label{prop:mm1}
In an M/M/1/FCFS system with~load~$\rho\equiv\lambda/\mu$, transmission threshold~$\theta\sim\expd(\alpha)$, and normalized transmission rate~$\eta\equiv\alpha/\mu$, we have~$W_i^{(s)}\sim\erld(i,\mu)$ and~$\widetilde{W}_i^{(s)}(\alpha)=1/(\eta+1)^i$ for all~$s\in\mathcal S$ and~$i\in\{1,2,\ldots,s\}$, while  \[\rnot=2\left(\frac{\rho}{1-\rho}\right)\left(\frac{\eta}{\eta+1-\rho}\right),\] which is convex increasing in $\lambda$, convex decreasing in $\mu$, and concave increasing in $\alpha$; the same properties also apply to $\lambda\rnot$.
\end{Proposition}

\begin{proof}{Proof of Proposition~\ref{prop:mm1}.}
This result is derived by combining traditional~M/M/1 analysis together with an examination of~$W_i^{(s)}$. Since the service policy is FCFS,~$W_i^{(s)}$ denotes the \emph{remaining} sojourn time of customer~$i$ given that there are~$s$ customers in the system when the infected customer arrives, and so,~$W_i^{(s)}\sim\erld(i,\mu)$ for all~$i\in\{1,2,\ldots,s\}$.  Hence, recalling that~$\eta\equiv\alpha/\mu$ \[\widetilde{ W}_i^{(s)}(\alpha)=\left(\frac{\mu}{\alpha+\mu}\right)^i=\left(\frac1{1+\eta}\right)^i,\] which together with the fact that~$\pi(s)=(1-\rho)\rho^s$ and~$\mathbb E[N]=\rho/(1-\rho)$ in an M/M/1 system, let us use Eq.~\eqref{eqn:r0exp} to prove the claim as follows:
\begin{align*}
\rnot&=2\left(\frac{\rho}{1-\rho}-(1-\rho)\sum_{s=0}^\infty\rho^s\sum_{i=1}^s\left(\frac{1}{\eta+1}\right)^i\right)\\
&=2\left(\frac{\rho}{1-\rho}-\left(\frac{1-\rho}{\eta}\right)\sum_{s=0}^\infty\rho^s\left(1-\left(\frac1{1+\eta}\right)^s\right)\right)
=2\left(\frac{\rho}{1-\rho}\right)\left(\frac{\eta}{\eta+1-\rho}\right).\end{align*}
The claims that $\rnot$ and $\lambda\rnot$ are convex increasing, convex decreasing, and concave increasing in $\lambda$, $\mu$, and $\alpha$, respectively, can be verified in a straightforward manner by taking first and second derivatives.
\hfill~$\square$
\end{proof}

\subsection{Finding~$\rnot$ for the M/M/c model}
\label{sec:mmc}
We now turn our attention to the~M/M/$c$ system. We again denote the system load by~$\rho$, although crucially, we now have~$\rho\equiv\lambda/(c\mu)$ as we are considering a~$c$-server system. Furthermore, we again denote each state by~$s$, the number of customers in the system (i.e.,~$n(s)=s$) and index customers in arrival order (e.g., in an M/M/2 system~$s=3$ denotes that two customers are in service and one is waiting in the queue, while in an M/M/4 system~$s=3$ denotes that three customers are in service and one server is free); the state space is~$\mathcal S=\{0,1,2,\ldots\}$. Unlike the M/M/1 system, customers in the first-come-first-serve M/M/$c$ system do not necessarily \emph{depart} in the order in which they arrive (FCFS only guarantees that customers \emph{enter service} in the order in which they arrive). As a result, we must consider three types of ``pairs'' that can be formed by the IC and a susceptible customer in position~$i$, based on the current system state:
\begin{enumerate}
    \item The case where~$1\le s< c$, i.e., when the IC finds at least one server serving a susceptible customer (customer~$i$) and at least one free server, allowing the IC to immediately enter service,
    \item The case where~$i\le c\le s$, i.e., when the IC finds all servers busy and joins the queue and we consider the overlap of its sojourn with that of a susceptible customer (customer~$i$) in service, and
    \item The case where~$c<i\le s$, i.e., when we have a situation like that in the previous case except for the fact that the susceptible customer (customer~$i$) is now in the queue.
\end{enumerate}
Careful analysis yields the~$W_i^{(s)}$ distributions, which can then be used in conjunction with standard M/M/$c$ analysis to obtain a closed-form expression for~$\rnot$ in terms of the Erlang-C formula,~$C(c,\rho)$, as presented in the following proposition:

\begin{Proposition}\label{prop:mmc}
Consider an M/M/$c/FCFS$ system with load~$\rho\equiv\lambda/(c\mu)$, transmission threshold~$\theta\sim\expd(\alpha)$, and normalized transmission rate~$\eta\equiv\alpha/\mu$.  In such a system, the sojourn time overlap between the IC and customer~$i$ has the following Laplace transform:
\[\widetilde{W}_i^{(s)}(\alpha)=\begin{cases}2/(\eta+2)& 1\leq s<c\\\left(\eta \left(\dfrac{c-1}{\eta+c}\right)^{s-c+1}+\eta+2\right)\left(\dfrac{1}{\eta^
2+3\eta+2}\right)& i\leq c\leq s\\[10pt]\left(\dfrac{c}{\eta+c }\right)^{i-c} \left(\eta \left(\dfrac{c-1}{\eta+c}\right)^{s-i+1}+\eta+2\right)\left(\dfrac{1}{\eta^
2+3\eta+2}\right)&c<i\leq s\end{cases},\] for all~$s\in\mathcal S$ and~$i\in\{1,2,\ldots,s\}$, while  \[\rnot=2\left(\left(\dfrac{\rho}{1-\rho}\right)C(c,\rho)+c\rho-\dfrac{1}{\eta +2 } \left(C(c,\rho)\left(\dfrac{2c\rho-c\eta}{\eta+c-c\rho}\right)+2c\rho \right)\right),\] where~$C(c,\rho)\equiv\dfrac{(c\rho)^c}{(1-\rho)c!}{\left(\displaystyle{\sum_{s=0}^{c-1}}{\dfrac{{(c\rho)}^{s}}{s!}+\dfrac{(c\rho)^c}{(1-\rho)c!}}\right)}^{-1}$ denotes the Erlang-C formula.
\end{Proposition}
\begin{proof}{Proof of Proposition~\ref{prop:mmc}.}
It is known that for an M/M/$c$ system  \[\pi(s)=\begin{cases}\dfrac{c!(1-\rho)}{s!(c\rho)^{c-s}}C(c,\rho)&0\le s\le c\\\dfrac{1-\rho}{\rho^{c-s}}C(c,\rho)&s>c,
\end{cases}\] and~$\mathbb E[N]=\dfrac{\rho}{1-\rho}C(c,\rho)+c\rho$ \citep[][chapter 14]{harchol2013performance}.  These two facts, together with the claimed values of~$\widetilde W_i^{(s)}(\alpha)$ and Eq.~\eqref{eqn:r0exp}, yield
\begin{align*}
    \rnot&=2\left(\mathbb E[N] -\sum_{s=0}^{\infty}{\pi(s)\sum_{i=1}^{s}{\widetilde W_i^{(s)}(\alpha)}}\right)\\
    &= 2\left(\mathbb E[N] -\sum_{s=0}^{c}{\pi(s)\sum_{i=1}^{s}{\widetilde W_i^{(s)}(\alpha)}}-\sum_{s=c+1}^{\infty}{\pi(s)\left(\sum_{i=1}^{c}{\widetilde W_i^{(s)}(\alpha)}+\sum_{i=c+1}^{s}{\widetilde W_i^{(s)}(\alpha)}\right)}\right)\\
    &= 2\left(\left(\dfrac{\rho}{1-\rho}\right)C(c,\rho)+c\rho-\dfrac{1}{\eta +2 } \left(C(c,\rho)\left(\dfrac{2c\rho-c\eta}{\eta+c-c\rho}\right)+2c\rho \right)\right),
\end{align*}
as claimed.  It remains only to prove that $\widetilde W_i^{(s)}(\alpha)$ is as claimed; we defer the proof of this claim to Appendix~\ref{app:ws}. \hfill~$\square$
\end{proof}

\section{Interventions}\label{sec:interventions}
In this section we consider three interventions aimed at reducing transmission risk in service facilities during a pandemic. They include limiting occupancy (section~\ref{sec:limitedOccupancy}), prioritizing high-risk customers (section~\ref{sec:priorities}), and increasing service rates (Section~\ref{sec:serviceRate}). These are some of the common recommendations by local governments and commonly practiced interventions by the service facilities during the COVID-19 pandemic. Our purpose in this section is to illustrate how the modeling paradigm laid out and elaborated upon in the previous two sections can help inform real-world decision making. While we share and comment upon insights that these illustrations reveal, the primary purpose of this section is not these insights in and of themselves, but rather an illustration of the type of insights that the paradigm presented in this paper can uncover, and more generally the type of questions it can help answer.  While in principle the interventions discussed here can be combined (i.e., implemented simultaneously), in the interest of brevity, we restrict attention to implementing these interventions one at a time.

\subsection{Intervention 1: Limited Occupancy}
\label{sec:limitedOccupancy}
Limiting the number of customers in business establishments was one of the highly practiced interventions during the COVID-19 epidemic, especially in grocery retailers~\citep{shumsky2020safeshopping}. For example, the New York State Department of Health guidance advises grocery retailers to limit their store occupancy, at any given time, to~50\% of their maximum capacity, inclusive of employees~\citep{NYSDOH}. Meanwhile, the United Food and Commercial Workers International Union made a recommendation to the Center for Disease Control and Prevention (CDC) to mandate grocery and drug stores to limit occupancy to~20\%--30\% of their maximum capacity~\citep{Redman2020UFCW}.

Therefore, our first intervention focuses on limiting the occupancy of the service facility. We study this intervention by modeling the service facility as an~M/M/$c$/$k$ system (a multi-server system with finite buffer), for all integer values the maximum \emph{system} occupancy~$k\ge c$; hence, the maximum queue length is~$K=k-c$. The analysis of~$\rnot$ for this model follows from a straightforward adaptation of the analysis presented in Sections~\ref{sec:mm1}~and~\ref{sec:mmc}, as the sojourn time overlap random variables~$W_i^{(s)}$ follow the same distributions as those followed by the analogous random variables in the~M/M/$c$ model (and they also agree with those that we found for the~M/M/1 model when~$c=1$, i.e.,~$W_i^{(s)}\sim\erld(i,\mu)$).  In this setting, for any given maximum system occupancy~$k$, we can evaluate~$\rnot$ in a closed form involving sums of finitely many terms.
\begin{Proposition}\label{prop:mmck}
In an~M/M/$c$/$k$/FCFS system with load~$\rho\equiv\lambda/(c\mu)<1$ and transmission threshold~$\theta\sim\expd(\alpha)$, 
\begin{align}
\rnot&=2\left(\sum_{s=0}^{c}{\pi(s)\sum_{i=1}^{s}{\widetilde W_i^{(s)}(\alpha)}}+\sum_{s=c+1}^{k-1}{\pi(s)\left(\sum_{i=1}^{c}{\widetilde W_i^{(s)}(\alpha)}+\sum_{i=c+1}^{s}{\widetilde W_i^{(s)}(\alpha)}\right)}\right),
\end{align}
where
\begin{align}
\pi(s)=\begin{cases}\dfrac{(c\rho)^{s}}{s!}{\left(\sum_{s=0}^{c}{\dfrac{(c\rho)^{s}}{s!}}+\dfrac{c^c}{c!}\sum_{s=c+1}^{k}{\rho^s}\right)}^{-1}&0\le s\le c\\\dfrac{c^{c}\rho^{s}}{c!}{\left(\sum_{s=0}^{c}{\dfrac{(c\rho)^{s}}{s!}}+\dfrac{c^c}{c!}\sum_{s=c+1}^{k}{\rho^s}\right)}^{-1}&c<s\leq k
\end{cases},
\label{eq:mmckPis}
\end{align}
and $\widetilde{W}_i^{(s)}(\alpha)$ is as given in Proposition~\ref{prop:mmc}.
\end{Proposition}
\begin{proof}{Proof of Proposition~\ref{prop:mmck}.}
The buffer size does not affect the distribution of the sojourn time overlap  under the FCFS policy. Therefore,~$\widetilde{W}_i^{(s)}(\alpha)$ remains the same as in the case of Proposition~\ref{prop:mmc}. Also, the steady-state probability distribution of an M/M/$c$/$k$ system is known to follow Eq.~\eqref{eq:mmckPis} \citep[][Chapter~3]{shortle2018queueing}. Using these in Eq.~\eqref{eqn:r0exp} we can establish the claimed result.\hfill$\square$
\end{proof}

After the imposition of occupancy limitations, some services including retailers were concerned about the impact of this mandate on their foot traffic and revenue~\citep{Pacheco2020}. Occupancy limitations have public health benefits and curb the spread of virus, but come at the cost of lost customers for some service facilities. As  an  example  of  how  our  methodology  can highlight this trade-off, Fig. 1 shows the trade-off between~$\rnot$ and the loss probability as the occupancy limit changes. The left figure shows three trade-off curves for various mean transmission thresholds~$1/\alpha$, whereas the right figure shows three trade-off curves for various system loads~$\rho$. Clearly, limiting the number of customers comes at the cost of turning away more customers. We see that more strict occupancy limits result in a larger reduction in~$\rnot$ for shorter mean transmission thresholds~$1/\alpha$ (left figure) and for higher system loads~$\rho$ (right figure).
\begin{figure}[]
\centering
\includegraphics[width=0.49\textwidth]{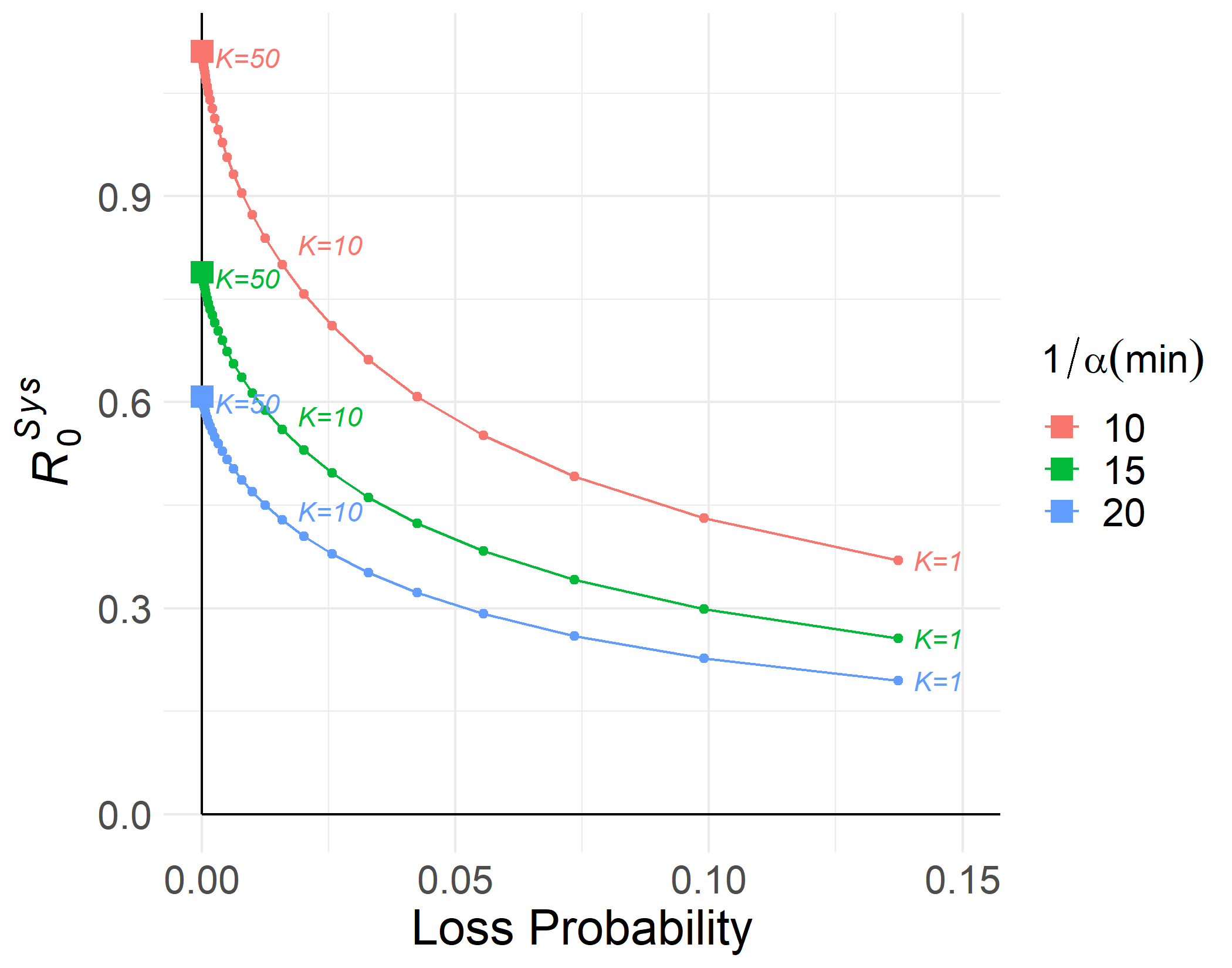}
\includegraphics[width=0.49\textwidth]{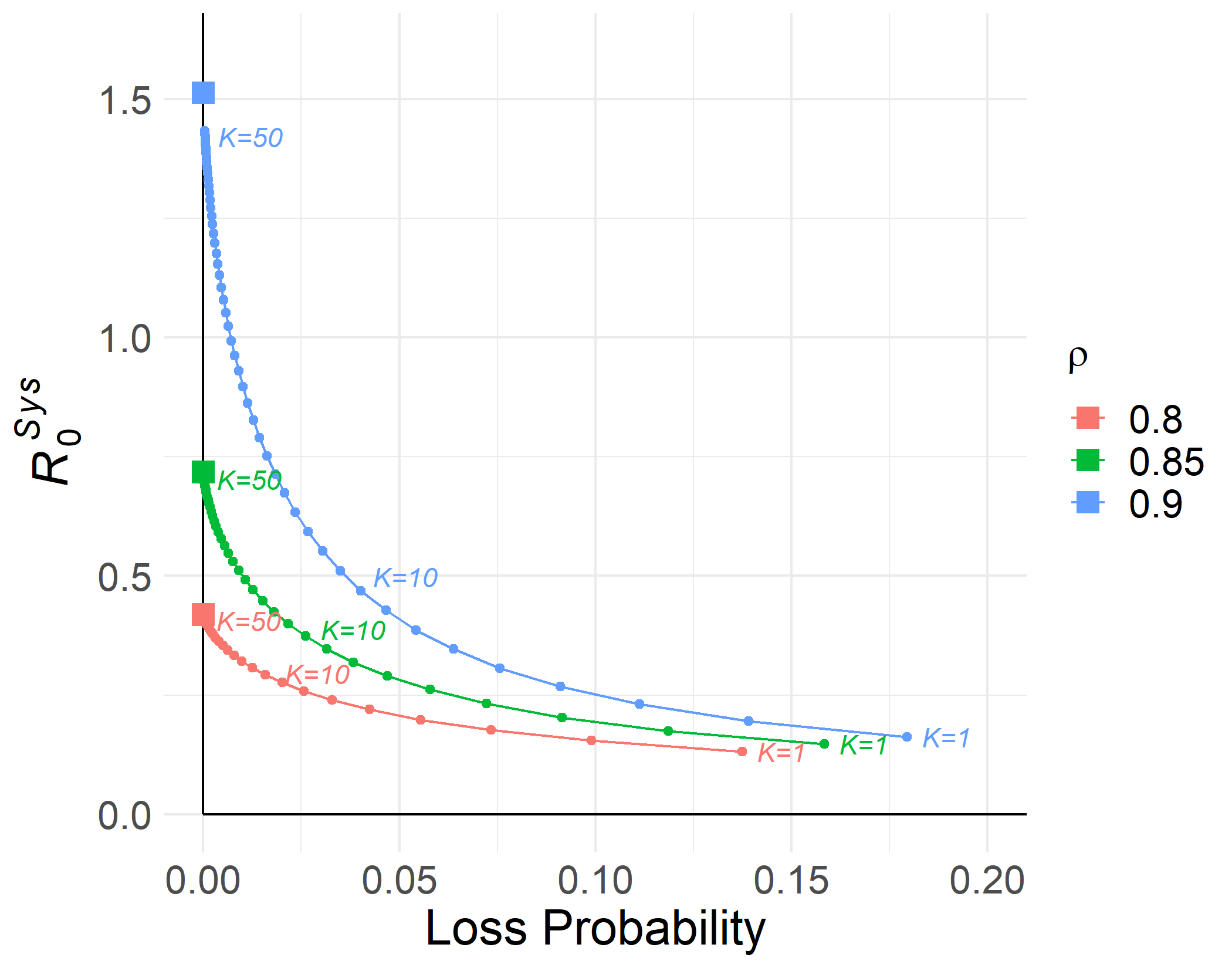}
\caption{The~$\rnot$ versus loss probability trade-off as the buffer size~$K = k-c$ changes. Left figure:~$\lambda=5$ arrivals per minute,~$c=2$, and~$\rho = 0.8$. Right figure:~$\lambda=5$ arrivals per minute,~$c=2$, and~$1/\alpha = 30$ minutes. The square marks on the vertical axis correspond to the case where~$K=\infty$ (i.e., the~M/M/$c$) system.}
\label{fig:tradeoff}
\end{figure}

We also observe that imposing a reasonable occupancy limit could reduce~$\rnot$ substantially at a small loss, compared to the infinite-buffer (M/M/$c$) system. For example, we can see that for system load~$\rho=0.9$ (right figure) as the occupancy limit is set at $k=12$ (i.e., $K=10$), $\rnot$ drops by almost $66\%$ (from~$\rnot=1.5130$ to~$\rnot=0.5104$) while the loss probability rises by less than~$5\%$. As the occupancy limit decreases, the reduction in~$\rnot$ comes at a much higher loss.

\subsection{Intervention 2: Protecting High-Risk Customers}
\label{sec:priorities}
Different individuals face different likelihoods of severe detrimental outcomes (e.g., serious illness, hospitalization, or death) once they become infected.  For example, older individuals and those with specific comorbidities such as hypertension or diabetes face much greater morbidity rates from COVID-19~\citep{Paudel2020,Sanyaolu2020, zhou2020clinical}. To this end, another intervention practiced during the COVID-19 pandemic by various service facilities including retailers, such as Walmart and Whole Foods, was to provide priority service for high-risk customers~\citep{Landry2020Amazon,Thakker2020}. Firms devised different strategies for this purpose. Examples from grocery retailers include reserving the first shopping hour for high-risk customers and prioritizing their curbside pickup orders~\citep{Amazon}.

Accordingly, we discuss interventions that aim to reduce the risk of infection for the high-risk population of customers. We formally consider two types of customers (high- and low-risk) arriving to the service facility.  We then measure the risk posed to each group under various system designs by introducing (and calculating) type-specific analogues of~$\rnot$: specifically, we let~$\rnothr$ (respectively,~$\rnotlr$) denote the expected number of high-risk (respectively low-risk) customers that an infectious customer will infect during their sojourn in the service facility. It follows that~$\rnot~=~\rnothr~+\rnotlr$.

We find these new risk measures in an M/M/1 system under two separate interventions.  First, we consider an intervention where the service facility serves high-risk customers with preemptive priority over their low-risk counterparts (Section~\ref{sec:preemptive}).\footnote{We opt for preemptive (rather than non-preemptive) priority in the interest of simplicity.  Relatively straightforward (if cumbersome) modifications will allow for extending our analysis to non-preemptive priority scheduling.} Next, we compare the previous approach to an idealized benchmark where the service facility is able to completely isolate high- and low-risk customers from one another, by giving each their own designated window in which to visit the facility  (Section~\ref{sec:DesignatingShoppingTimes}).

\subsubsection{Preemptive Priority for High-risk Customers.}
\label{sec:preemptive}
We consider an M/M/1 system with arrival, service, and transmission rates $\lambda$, $\mu$, and $\alpha$, respectively, where each arrival is of type-$\hsf$ (high-risk) or of type-$\lsf$ (low-risk).  An arrival is of type $\tsf\in\{\hsf,\lsf\}$ with probability $\qt$ (independent of all arrival times, service requirements, and transmission thresholds)\footnote{For simplicity, we assume that high-risk customers are more likely to suffer adverse effects once infected, but they are otherwise identical to their low-priority counterparts, i.e., they are not more (or less) likely to contract an infection (over a fixed duration of time), nor are they more (or less) likely to be infectious when they arrive.}, so that~ $\qh+\ql=1$. For convenience, we let $\lambda_{\tsf}\equiv \qt\lambda$ be the arrival rate of type-$\tsf$ customers and $\rho_{\tsf}\equiv\lambda_{\tsf}/\mu$ be the contribution of such customers to the system load.  Note that it follows from the Poisson splitting property that type-$\tsf$ customers arrive according to a Poisson process with rate $\lambda_{\tsf}$~\citep[][Chapter~11.7]{harchol2013performance}.

We proceed to analyze this system under an intervention where type-$\hsf$ customers are given preemptive priority over their type-$\lsf$ counterparts (within each type, customers are scheduled~FCFS).  Specifically, we seek to find~$\rnot$ (defined as before), along with~$R_0^{\tsf}$---the expected number of type-$\tsf$ customers that an arbitrary IC (i.e., an infectious customer who is of type-$\hsf$ with probability~$\qh$ and of type-$\lsf$ otherwise) will infect, assuming that all other customers are susceptible.  

In order to compute the values of interest, we introduce some useful auxiliary notation: for any types $\tsfo,\tsft\in\{\hsf,\lsf\}$ let $\rtb$ (respectively, $\rta$) denote the expected number of type-$\tsft$ customers that a type-$\tsfo$ IC will infect during their sojourn in the facility from among those type-$\tsft$ customers who arrived to the system \emph{before} (respectively, \emph{after}) the IC arrived to the system (under the usual assumption that all customers other than the IC are susceptible).  Leveraging our new metrics (and notation), we have the following result:
\begin{Proposition}\label{prop:priority1}
In the M/M/1 system with preemptive priorities described above, we have
\begin{align}
    \rnot&=2\left(\qh\left(\rhhb+\rhlb\right)+\ql\left(\rlhb+\rllb\right)\right)\label{eq:1}\\
    \rnothr&=2\qh\left(\frac{\rho_H}{1-\rho_H}\right)\left(\frac{\eta}{\eta+1-\rho_H}\right)+\ql(\rlhb+\rlha)\\
    \rnotlr&=\rnot-\rnothr,
\end{align}
where expressions for $\rhhb$, $\rhlb$, $\rlhb$, $\rllb$, and $\rlha$ together with their derivations are given (in terms of the limiting probability distribution of the M/M/1 system with two priority classes) in Appendix~\ref{app:priority}.
\end{Proposition}
\begin{proof}{Proof of Proposition~\ref{prop:priority1}}
The first equation follows from the symmetry argument that for the whole system, the expected number of susceptible customers (SCs) that the IC infects among those who arrive before and after the IC are the same. Conditioning on the type of IC, either high-risk or low-risk SCs who were present in the system when the IC arrives will possibly be infected which yields Eq.~\ref{eq:1}. Next, we obtain the second equation by conditioning on the type of the IC. If the IC is high-risk, then applying Proposition~\ref{prop:mm1} with the load $\rho_\hsf$ gives the first half of the equation. Otherwise, the IC will be low-risk and infect $\rlhb+\rlha$ high-risk SCs on average. Finally, the last equation follows from the fact that~$\rnot=\rnothr+\rnotlr$.\hfill$\square$
\end{proof}

\subsubsection{Designated Time Windows for High-Risk Customers.}
\label{sec:DesignatingShoppingTimes}
We now consider an alternative intervention (once again) designed to protect high-risk customers: type-$\hsf$ and type-$\lsf$ customers have their own pre-designated time windows for visiting (and being served in) the service facility; customers of each type are served in FCFS order during their type's window (and neither admitted nor served outside of that window).  We make two idealistic assumptions about this intervention: (i) we have complete compliance, that is customers only arrive during their own designated time windows and no potential arrivals are lost (i.e., all customers can adapt their schedules as necessary to visit the service facility at the same rate during the windows designated for their type); (ii) the time windows are sufficiently long so that we can treat each customer as seeing the system in steady-state (i.e., a steady-state specific to their type/window) upon their arrival; among other things, this implies that we ignore the effects of the boundaries between the time windows.  We further assume that the arrival processes remain Poisson (within each given window).  While these assumptions are not very realistic, they represent an ideal case where intervention works as intended.  In any case, this intervention represents a potentially useful baseline for comparing with the efficacy of the alternative intervention discussed in Section~\ref{sec:preemptive} (prioritization).

This intervention requires one to set a decision variable~$f$, which denotes the fraction of time that the system is open to (i.e., will receive arrivals of and serve) type-$\hsf$ customers; the remaining~$1-f$ fraction of time, the system will be open to type-$\lsf$ customers instead. Therefore, during each time window for a given customer type, that type's arrival rate will scale up, i.e., type-$\hsf$ and~$\lsf$ customers arrive at rates~$\lhr^\mathsf{win.}=(1/f)\lhr$ and $\llr^\mathsf{win.}=(1/{(1-f))}\llr$ during their respective windows; we assume that $f$ is chosen so as to guarantee system stability (i.e.,~$\lhr^\mathsf{win.}<f\mu$ and~$\llr^\mathsf{win.}<(1-f)\mu$).  Based on our idealistic assumptions, we can otherwise treat type-$\tsf$ customers as having their own type-specific M/M/1 system with load~$\rho_{\tsf}^{\mathsf{win.}}\equiv\lambda_{\tsf}^{\mathsf{win.}}/\mu$.  Once we observe that an IC can only infect customers of their own type in this setting, applying Proposition~\ref{prop:mm1} yields the following:
\[R_0^{\tsf}=2\qt\left(\dfrac{\rho_{\tsf}^{\mathsf{win.}}}{1-\rho_{\tsf}^{\mathsf{win.}}}\right)\left(\dfrac{\eta}{\eta+1-\rho_{\tsf}^{\mathsf{win.}}}\right), \quad \tsf\in\{\hsf,\lsf\}.\]

As an example of how our methodology can allow for a comparison of interventions aiming at protecting high-risk customers, Fig.~\ref{fig:RHRL} shows the trade-off between $\rnothr$ and $\rnotlr$ as we vary the $f$ parameter for the dedicated time window intervention; also shown in Fig.~\ref{fig:RHRL} are the $\rnothr$ and $\rnotlr$ values associated with the prioritization of high-risk customers.  Note that since the overall arrival rate $\lambda$ is constant across these interventions, the~$\rnothr$ and~$\rnotlr$ metrics are proportional to the rate at which high- and low-risk customers become infected, respectively (assuming a low background infection rate, $p$).  
\begin{figure}[]
    \centering
%
%
\resizebox{0.5\textwidth}{!}{
\definecolor{mycolor1}{rgb}{0.00000,0.44700,0.74100}%
\definecolor{mycolor2}{rgb}{0.85000,0.32500,0.09800}%
\definecolor{mycolor3}{rgb}{0.92900,0.69400,0.12500}%
\begin{tikzpicture}

\begin{axis}[%
width=4.521in,
height=3.566in,
at={(0.758in,0.481in)},
scale only axis,
xmin=0,
xmax=3,
xtick={0,0.5,1,1.5,2,2.5,3},
xticklabels={,0.5,1,1.5,2,2.5,3},
xlabel style={font=\color{white!15!black}},
xlabel={\Large $\rnothr$},
ymin=0,
ymax=3,
ylabel style={font=\color{white!15!black}},
ylabel={\Large $\rnotlr$},
axis background/.style={fill=white},
legend style={at={(0.9,0.9)},legend cell align=left, align=left, draw=white!15!black}
]
\addplot [color=mycolor1,line width=2.0pt]
  table[row sep=crcr]{%
2.57575757575758	0.561739130434783\\
2.41668235737149	0.575373374879123\\
2.27304714989444	0.589580873671782\\
2.14285714285714	0.604395604395604\\
2.02443550881534	0.619854151964243\\
1.91636363636364	0.635995955510617\\
1.81743421052632	0.652863583235932\\
1.72661405097684	0.670503038924091\\
1.64301443601664	0.688964104391199\\
1.56586722280153	0.708300722802817\\
1.4945054945055	0.728571428571428\\
1.42834777163136	0.749839830466709\\
1.36688505062537	0.772175155659895\\
1.30967010114378	0.79565286371569\\
1.25630857966835	0.820355341083569\\
1.20645161290323	0.84637268847795\\
1.1597895775111	0.873803615739098\\
1.1160468590298	0.902756461415677\\
1.07497741644083	0.933350357507659\\
1.03636101291393	0.965716563683474\\
1	0.999999999999997\\
0.965716563683479	1.03636101291393\\
0.933350357507663	1.07497741644083\\
0.902756461415682	1.11604685902979\\
0.873803615739102	1.15978957751109\\
0.846372688477954	1.20645161290322\\
0.820355341083573	1.25630857966834\\
0.795652863715693	1.30967010114377\\
0.772175155659898	1.36688505062537\\
0.749839830466712	1.42834777163135\\
0.728571428571431	1.49450549450549\\
0.70830072280282	1.56586722280152\\
0.688964104391202	1.64301443601663\\
0.670503038924094	1.72661405097683\\
0.652863583235935	1.8174342105263\\
0.635995955510619	1.91636363636362\\
0.619854151964246	2.02443550881533\\
0.604395604395606	2.14285714285713\\
0.589580873671785	2.27304714989442\\
0.575373374879125	2.41668235737147\\
0.561739130434785	2.57575757575755\\
};
\addlegendentry{Varying $f$}


\end{axis}
\filldraw [red] ((1.604in,1.93in) circle (2.5pt);
\filldraw [mycolor2] (2.241in,1.64in) rectangle ++(5pt,5pt);
\draw[->,dashed,thick] (1.45in,1.65in) -- (1.57in,1.875in);
\node[right, align=left]
at (1.13in,1.4in) {Priority \\[-12pt]Scheduling};
\draw[->,dashed,thick] (2.55in,1.85in) -- (2.34in,1.71in);
\node[right, align=left]
at (2.6in,1.9in) {FCFS \\[-12pt]Scheduling};
\end{tikzpicture}%
}
    \caption{The trade-off between $\rnothr$ and $\rnotlr$ of designated time windows as~$f$ varies ($f\in[15/34,19/34]$) in an~M/M/1 system with~$\lhr=\llr=1.5$,~$\mu=4$, and~$\alpha=0.5$.}
    \label{fig:RHRL}
\end{figure}
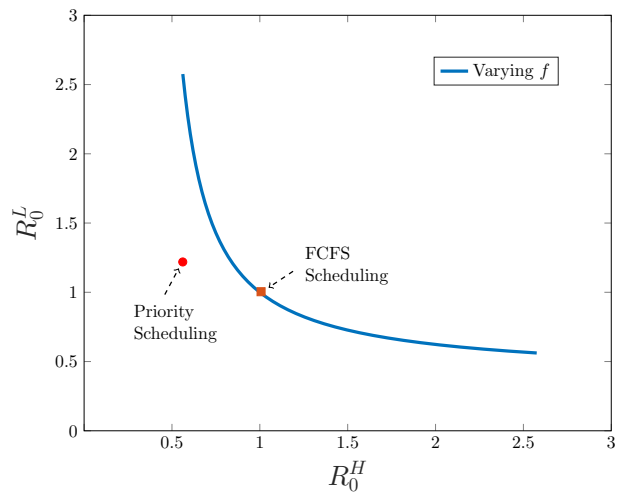

In this setting, the ordinary (no-intervention) FCFS scheduling policy obtains the same infections risks as setting $f=1/2$, which is the value of $f$ that minimizes $\rnot=\rnothr+\rnotlr$ under the designated time window intervention.  Since we have chosen a setting where high- and low-risk customers are equally numerous for illustrative purposes,\footnote{Other parameter choices, including making high-risk customers substantially rarer than their low-risk counterparts, yield qualitative similar plots.} giving one group a greater share of access to the service facility will lead to an \emph{overall} increase in infections.  That is, each high-risk customer that is saved from infection by increasing $f$ above $1/2$ comes at the cost of more than one low-risk customer being infected, in expectation.

Remarkably, under the other intervention (where we let both types of customers access the service facility at any time while prioritizing high-risk customers), we can earn a ``free lunch'': $(\rnothr,\rnotlr)=(0.561,1.221)$, and this point falls ``below'' the ``Pareto risk frontier'' associated with the designated time window intervention.  Moreover, $\rnot=\rnothr+\rnotlr=1.782$ under this intervention, as opposed to $\rnot=2$ under FCFS.  This suggests that prioritizing high-risk customers helps both those customers and the overall system.  That is, each high-risk customer that is saved from infection due to prioritization comes at an expected cost of \emph{less} than one additional infection among low-risk customers.  This result is driven by the exponential dose-response model as transmission thresholds are \emph{memoryless}: under this model, even when all customers are of the same risk-level, it is actually preferable to serve newly arriving customers ahead of those who have already been waiting in the system, as the latter may have already been infected.  As discussed earlier, there are other dose-response models, so we caution that in systems where one of these models turns out to be more appropriate than the exponential dose-response model, the benefits to prioritization over other interventions may be somewhat diminished.

\subsection{Intervention 3: Increasing the Service Rate}
\label{sec:serviceRate}
It is well documented that the duration of time that customers spend inside service facilitates increases the infection risk significantly \citep{rea2007duration}. It is advised that people who live in areas with high rates of COVID-19 (and its new strains) should limit their time inside stores to no more than 15 minutes~\citep{JacksonBreen2020ShoppingTime}. Though mandating customers to keep their time inside service facilitates short is difficult to implement, the recommendations and signals from healthcare organizations and service managers could impact customers' behavior. Empirical evidence (from previous epidemics as well as the COVID-19 pandemic) suggests that customers shop more quickly and are more reluctant to socialize inside grocery stores during an epidemic~\citep{dawson1990shopping,szymkowiak2020store, wang2020covid}. Managing the activities to speed up the service, like setting up self-checkout machines and appropriate staffing, can also shorten the sojourn time.

We capture these in the final intervention that we consider by increasing an M/M/1 system's service rate~$\mu$. While this may be accomplished by simply having staff work faster, in a retail store setting one can also limit the number of items a customer could buy, which reduces the amount of time each customer spends in the system, effectively increasing~$\mu$. The left plot in Fig.~\ref{fig:ISR} shows that $\lambda\rnot$ (which is proportional to the rate of infections) decreases rapidly as the service rate $\mu$ increases by a factor of $\mathcal{k}$. However, if customers are able to procure fewer goods with each visit to the facility, one would anticipate that they would shop more frequently and the arrival rate~$\lambda$ would also increase in response. Therefore, in the right plot of Fig.~\ref{fig:ISR}, we calculate the maximum factor $\mathcal{k}^{\max}$ that $\lambda$ could possibly increase by in order to maintain the same risk level (i.e. $\lambda\rnot$) of the system as the original scenario where the service rate $\mu=1$. Hence, the region below the blue curve in the plot shows the feasible scaling factors of $\lambda$ and $\mu$ such that the risk ($\lambda\rnot$) can be reduced. The effect of the decreased maximum $\lambda$ scale-up (i.e., the gap between the 45-degree line and the blue curve) will become more pronounced when the transmission rate $\alpha$ is larger.
\begin{figure}[]
\centering
%
%
\resizebox{0.45\textwidth}{!}{
\definecolor{mycolor1}{rgb}{0.00000,0.44700,0.74100}%
\begin{tikzpicture}

\begin{axis}[%
width=4.521in,
height=3.566in,
at={(0.758in,0.481in)},
scale only axis,
xmin=0.5,
xmax=3,
xtick={0.5,1,1.5,2,2.5,3},
xticklabels={,1,1.5,2,2.5,3},
xlabel style={font=\color{white!15!black}},
xlabel={\Huge $\mathcal{k}\mu$},
ymin=0,
ymax=40,
ylabel style={font=\color{white!15!black}},
ylabel={\Large $\lambda\rnot$},
axis background/.style={fill=white},
xmajorgrids,
ymajorgrids,
legend style={legend cell align=left, align=left, draw=white!15!black}
]
\addplot [color=mycolor1,line width=2.0pt]
  table[row sep=crcr]{%
1	36.1904761904762\\
1.01	29.8909657320872\\
1.02	25.3997378768021\\
1.03	22.0382882882883\\
1.04	19.4296951819076\\
1.05	17.3478260869565\\
1.06	15.6487956487956\\
1.07	14.2366946778711\\
1.08	13.0451366815003\\
1.09	12.0267131242741\\
1.1	11.1466666666667\\
1.11	10.378937007874\\
1.12	9.70360237118102\\
1.13	9.10517387616625\\
1.14	8.57142857142857\\
1.15	8.09259259259259\\
1.16	7.66075773375044\\
1.17	7.26945716154349\\
1.18	6.91335183472094\\
1.19	6.58799533799534\\
1.2	6.28965517241379\\
1.21	6.01517530088958\\
1.22	5.76186925180214\\
1.23	5.52743614001892\\
1.24	5.30989407257156\\
1.25	5.10752688172043\\
1.26	4.91884117526197\\
1.27	4.74253144654088\\
1.28	4.57745153397327\\
1.29	4.42259112233851\\
1.3	4.27705627705628\\
1.31	4.14005322687957\\
1.32	4.01087478010555\\
1.33	3.88888888888889\\
1.34	3.77352897584111\\
1.35	3.66428571428571\\
1.36	3.5607000137798\\
1.37	3.46235700984304\\
1.38	3.36888089425671\\
1.39	3.2799304520616\\
1.4	3.1951951951952\\
1.41	3.11439200186003\\
1.42	3.03726218619836\\
1.43	2.96356893542757\\
1.44	2.89309506185894\\
1.45	2.82564102564103\\
1.46	2.76102319100229\\
1.47	2.69907228449942\\
1.48	2.63963202853656\\
1.49	2.58255792738551\\
1.5	2.52771618625277\\
1.51	2.47498274672188\\
1.52	2.42424242424242\\
1.53	2.37538813531623\\
1.54	2.32832020370812\\
1.55	2.28294573643411\\
1.56	2.2391780614943\\
1.57	2.19693622035646\\
1.58	2.15614450908569\\
1.59	2.11673206278027\\
1.6	2.07863247863248\\
1.61	2.04178347350153\\
1.62	2.00612657237828\\
1.63	1.971606824548\\
1.64	1.93817254462897\\
1.65	1.90577507598784\\
1.66	1.87436857431509\\
1.67	1.84390980939098\\
1.68	1.81435798328881\\
1.69	1.78567456345234\\
1.7	1.7578231292517\\
1.71	1.73076923076923\\
1.72	1.70448025869713\\
1.73	1.67892532434365\\
1.74	1.65407514884675\\
1.75	1.62990196078431\\
1.76	1.60637940145074\\
1.77	1.58348243714097\\
1.78	1.56118727784702\\
1.79	1.53947130182872\\
1.8	1.51831298557159\\
1.81	1.49769183869001\\
1.82	1.47758834337478\\
1.83	1.4579838980208\\
1.84	1.43886076470346\\
1.85	1.42020202020202\\
1.86	1.40199151029476\\
1.87	1.38421380707496\\
1.88	1.36685416905828\\
1.89	1.34989850387189\\
1.9	1.33333333333333\\
1.91	1.31714576074332\\
1.92	1.30132344023116\\
1.93	1.28585454800477\\
1.94	1.27072775536939\\
1.95	1.25593220338983\\
1.96	1.24145747908124\\
1.97	1.22729359302249\\
1.98	1.21343095829436\\
1.99	1.19986037065245\\
2	1.18657298985168\\
2.01	1.17356032204536\\
2.02	1.16081420318785\\
2.03	1.14832678337502\\
2.04	1.13609051206144\\
2.05	1.12409812409812\\
2.06	1.11234262653821\\
2.07	1.10081728616211\\
2.08	1.08951561767706\\
2.09	1.07843137254902\\
2.1	1.06755852842809\\
2.11	1.05689127913108\\
2.12	1.04642402514743\\
2.13	1.03615136463721\\
2.14	1.02606808489161\\
2.15	1.01616915422886\\
2.16	1.00644971429973\\
2.17	0.996905072779148\\
2.18	0.987530696421333\\
2.19	0.97832220445782\\
2.2	0.96927536231884\\
2.21	0.960386075659851\\
2.22	0.951650384676127\\
2.23	0.943064458689459\\
2.24	0.93462459099194\\
2.25	0.926327193932828\\
2.26	0.918168794235252\\
2.27	0.91014602853043\\
2.28	0.902255639097745\\
2.29	0.894494469799761\\
2.3	0.886859462201928\\
2.31	0.879347651867286\\
2.32	0.871956164817123\\
2.33	0.86468221414899\\
2.34	0.857523096804058\\
2.35	0.850476190476191\\
2.36	0.843538950655605\\
2.37	0.836708907800364\\
2.38	0.829983664629334\\
2.39	0.823360893530606\\
2.4	0.816838334079713\\
2.41	0.810413790662277\\
2.42	0.804085130196037\\
2.43	0.797850279947466\\
2.44	0.791707225438462\\
2.45	0.785654008438818\\
2.46	0.779688725040452\\
2.47	0.773809523809524\\
2.48	0.768014604012844\\
2.49	0.762302213915117\\
2.5	0.756670649143767\\
2.51	0.751118251118251\\
2.52	0.745643405540934\\
2.53	0.740244540946749\\
2.54	0.734920127308998\\
2.55	0.729668674698795\\
2.56	0.72448873199577\\
2.57	0.71937888564778\\
2.58	0.714337758477478\\
2.59	0.709364008533702\\
2.6	0.70445632798574\\
2.61	0.699613442058633\\
2.62	0.694834108007761\\
2.63	0.690117114131035\\
2.64	0.685461278817115\\
2.65	0.680865449628127\\
2.66	0.676328502415459\\
2.67	0.671849340467235\\
2.68	0.667426893686183\\
2.69	0.663060117796632\\
2.7	0.658747993579454\\
2.71	0.654489526133821\\
2.72	0.650283744164685\\
2.73	0.64612969929495\\
2.74	0.642026465401359\\
2.75	0.637973137973138\\
2.76	0.633968833492511\\
2.77	0.630012688836218\\
2.78	0.626103860697225\\
2.79	0.622241525025824\\
2.8	0.618424876489393\\
2.81	0.614653127950082\\
2.82	0.61092550995975\\
2.83	0.607241270271491\\
2.84	0.603599673367115\\
2.85	0.6\\
2.86	0.596441546752719\\
2.87	0.592923625608907\\
2.88	0.589445563538828\\
2.89	0.586006702098141\\
2.9	0.582606397039387\\
2.91	0.579244017935716\\
2.92	0.575918947816428\\
2.93	0.572630582813882\\
2.94	0.569378331821379\\
2.95	0.566161616161616\\
2.96	0.562979869265343\\
2.97	0.559832536359848\\
2.98	0.556719074166937\\
2.99	0.553638950610065\\
3	0.550591644530307\\
};

\end{axis}
\end{tikzpicture}%
}
\hspace{30pt}
\centering
\input{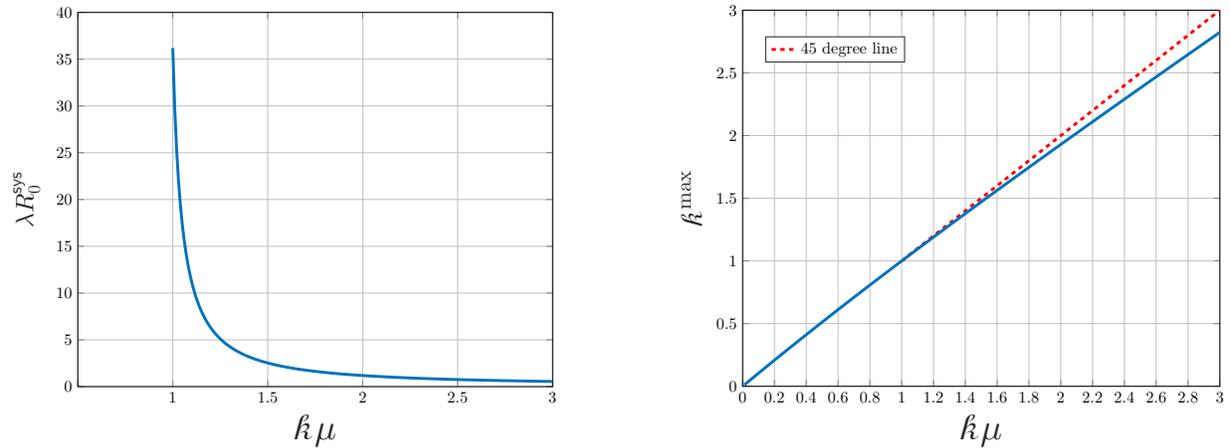}
\caption{M/M/1/FCFS system with service rate increase when $\lambda=0.95$, $\mu=1$, and~$\theta\sim\expd(1)$}
\label{fig:ISR}
\end{figure}

We note that plots such as the one depicted on the right in Fig.~\ref{fig:ISR} also facilitate the study of the impact of a ``reverse'' of the phenomena discussed above.  Customers may be discouraged from shopping regularly, therefore reducing the arrival rate~$\lambda$, which causes them to actually spend longer periods of time in the service facility, thereby also reducing the service rate~$\mu$ (i.e., they need to procure more goods in each visit).  Based on the scaling of each factor, one can examine if this collective behavioral shift increases or decreases in the risk of infection at the service facility.

\section{Model Modifications and Concluding Remarks}\label{sec:extensions}
Our modelling framework for measuring the risk of infection in small-scale settings in the presence of stochastic arrivals and departures can be modified in various ways to capture some more nuanced features of disease transmission. In this section, we discuss how our novel~$\rnot$ measure can be modified to capture these features, such as heterogeneity in individuals infectiousness and susceptibility and the impact of spatial dynamics on transmissions.  We then proceed to present our concluding remarks.

\subsection{Heterogeneity in Infectiousness and Susceptibility}\label{sec:extension-het}

In reality individuals exhibit different levels of infectiousness and susceptibility~\citep{gomes2020individual}.  While individual heterogeneity may be due to differences in human bodies, it can also be a result of behavioral interventions.  For example, research has provided ample evidence that when an infectious individual wears a mask properly, the risk that they transmit a SARS-CoV-2 infection to others (over a short time duration) is significantly reduced \citep{lai2012effectiveness,bundgaard2020effectiveness,ngonghala2020mathematical}; on the other end, evidence also suggests that susceptible individuals are less vulnerable to becoming infected when they wear masks properly \citep{wu2004risk}. Other behaviors can also affect transmission rates, e.g., infectious individuals expel more aerosols the louder they talk~\citep{tang2013airflow, Buonanno2020}.  Perhaps the starkest form of heterogeneity comes from the varying levels of immunity present in the population: recently infected and vaccinated individuals are far less likely to become infected than individuals who have no immunity~\citep{Polack2020,Harvey2021}.

While modeling the transmission threshold~$\theta$ following an exponential distribution for each pair of infectious-susceptible customer may be seen as accounting for this variability, one could argue that the exponential distribution captures \emph{idiosyncratic}, rather than \emph{systematic}, variation introduced by heterogeneity.  We would naturally anticipate markedly different transmission dynamics in environments where roughly half of the susceptible individuals are vaccinated and half are not, than one where nearly all susceptible individuals exhibit the same level of immunity (if any). Modifying the transmission rate~$\alpha$ to account for the level of vaccination (or mask usage) in the population may fail to capture important nuanced transmission dynamics.  To this end, we can allow for the transmission threshold~$\theta$ for each infectious-susceptible pair to be drawn from a \emph{hyperexponential} distribution (i.e., to be drawn from one of several exponential distributions, each with its own probability).  Such hyperexponential dose response models have been proposed in other contexts (e.g., in models of bio-catalytic reactions) in the literature~\citep{kuhl2006receptor, kuhl2007nonexponential}. 

Consider an example where customers wear masks with probability $q$ and the likelihood of wearing a mask is the same whether one is infected or not. In such a case, we may allow for four transmission rates~$\alpha_{\mathrm{MM}}, \alpha_{\mathrm{MU}}, \alpha_{\mathrm{UM}}$, and~$\alpha_{\mathrm{UU}}$ respectively denoting transmission rates from \emph{masked} infectious to \emph{masked} susceptible customers, \emph{masked} infectious to \emph{unmasked} susceptible customers, and so on.  Then, for any given infectious-susceptible customer pair, the corresponding transmission threshold follows a hyperexponential distribution with rates~$\alpha_{\mathrm{MM}}, \alpha_{\mathrm{MU}}, \alpha_{\mathrm{UM}}$, and~$\alpha_{\mathrm{UU}}$ with respective probability~$q^2, q(1-q), q(1-q)$, and~$(1-q)^2$. Using hyperexponential transmission thresholds does not substantially complicate the computation of~$\rnot$. We provide the derivations of $\rnot$ in this case in Proposition~\ref{prop:hyper} in Appendix~\ref{hyper}.

This approach could also be taken to account for multiple viral variants, each with their own transmission rate. However, a better approach might be to model each variant separately (and in parallel) with their own $\rnot$ value, as we may also be interested in tracking the spread of (and risk of becoming infected with) each such variant at a service facility separately.

\subsection{Multiple Simultaneous Infectors}\label{sec:extension-groups}

Our transmission model and our~$\rnot$ metric (and metrics derived from it such as~$p \lambda \rnot$) are suitable for settings where infection rate is very low, i.e., specifically when at any given time there is at most one infectious customer in the service facility.  We can adapt the transmission model to accommodate the possibility of multiple simultaneous infectious customers (ICs) by assuming that the SC becomes sick if their \emph{cumulative exposure} to infectious \emph{customers} exceeds a given (e.g., exponentially distributed) threshold.  Here, by cumulative exposure to ICs we mean the sum of the SC's sojourn time overlap with each potential IC. 

For example, spending $\tau$ units of time in the system while exactly $m$ ICs were also present in the system contributes $m\tau$ to this cumulative exposure.  Formally, an SC who arrived to the system at time $a$ and departed at time $d$ becomes infected with probability~$1-\exp\left(-\int_a^d \iota(t)\,dt\right),$ where $\iota(t)$ denotes the number of ICs present in the system at time $t$.  We can justify this because an SC becomes infected as soon as any IC transmits the infection to that SC; so, if the time it takes \emph{any given} IC to infect the SC is exponentially distributed (and independent of all other such infection times), then the time it takes for \emph{any one} IC to infect the SC is exponentially distributed with a rate equal to the sum of these rates.  This yields the model described above.

This modified model allows one to compute the average rate of infections without having to rely on the $p\lambda\rnot$ (or $\lambda p(1-p)\rnot$) assumption which suffers when $p\gg0$.  That said, this model significantly complicates the analysis of the system, but we conjecture that tractable analytical results are still obtainable, suggesting a potential avenue for future work.

One limitation of the single-infector assumption, however, may have a simple resolution.  In reality, many people enter service facilities in small groups formed of members of the same household. If one member of such a group is infectious, there is a substantial likelihood that the group consists of more than one infectious customer.  However, such groups often operate as a single customer, so we can treat the fact that some groups are more likely to transmit infections (due to having multiple infectious members) as simply an instance of heterogeneity, which we can treat with. 

An alternative and more sophisticated (yet technically challenging) method of allowing for multiple simultaneous ICs in our transmission model is to separately track the concentration of viral particles in the environment (e.g., as a fluid quantity).  In such a setting each IC would contribute to such a concentration at a constant rate (by expelling aerosols while breathing) throughout their sojourn in the service facility; this concentration would also be subject to exponential decay.  An SC would then become infected if the cumulative concentration level that are exposed to throughout their sojourn exceeds a given (e.g., exponentially distributed) threshold.  While such an approach is far beyond the scope of this paper, it is particularly attractive as it allows for the potential for an IC to \emph{indirectly} infect an SC when the latter arrives to the service facility shortly after the former departs (i.e., this approach permits infections even when there is no sojourn time overlap between the IC and SC).

\subsection{Impact of Spatial Positioning on Transmission}\label{sec:extension-distance}
While the spatial dynamics of customers and the direction of airflow, etc. are likely to play an important role in viral transmission, in many circumstances incorporating these dynamics is likely to require analytic techniques that are beyond the scope of this paper (and would be likely to hinder tractability) \citep[for example, see the processes in Section~4 of][]{zhang2015} and necessitate factoring in idiosyncrasies of a particular system \citep[see][for examples]{Wallace2002}. 

Nevertheless, physical distancing between customers is one of the most commonly employed interventions, and we present a generalization of our transmission model that is rich enough to capture a variety of setting-specific idiosyncrasies. Specifically, we modify our transmission model to incorporate distances between customers by allowing transmission rates to be distance-dependent.  Most generally, we consider \emph{position-dependent} transmission rates as follows: we can think of each customer as occupying a \emph{position} in space, with customers occasionally moving from one position to another.  For example, in the context of~M/M/1/FCFS or~M/M/1/$k$/FCFS systems, these positions could simply be the queue positions~$1,2,3,\ldots$, i.e., a customer is in position~$m$ if it will complete service after the~$m$-th departure.  In this setting, a newly arriving customer occupies the lowest-numbered unoccupied position.  Whenever the system is nonempty, the customer at position~$1$ is in service; when they depart for each $m\in\{1,2,3,\ldots\}$ the customer in position $m$ (if any) advances (i.e., moves to) position $m-1$.  This position model can easily be extended to a variety of other queueing systems (M/M/$c$, M/M/$c$/$k$) and scheduling disciplines.

Once we have defined positions, let~$\alpha_{m,j}$ denote the rate at which an IC at position~$m$ transmits an infection to a SC at position~$j$.  The transmission rates~$\alpha_{m,j}$ may depend on a variety of factors (the physical distance between those positions, the environmental conditions near and between those positions including the direction and speed of airflow, etc.). 

We briefly discuss the generality afforded by position-dependent transmission rates.  Naturally, we may assume that $\alpha_{m,j}$ depends on the distance between positions $m$ and $j$ in physical space, e.g., with the ``gravity model'' \citep{Jia2020, Balcan2009}: $\alpha_{m,j}=1/{||x_m-x_j||}^2$, where~$x_m$ and~$x_j$ denote the locations of positions $m$ and $j$ in the Euclidean plane, respectively.  Note that in e.g., a ``snake queue,'' the distance between $x_m$ and $x_j$ may be non-monotonic in $|m-j|$; for example, in Fig.~\ref{fig:snake}, position~1 is~6 feet away from position~8 but~18 feet away from position~4.  However, factors other than distance may impact transmission rates, with Fig.~\ref{fig:snake} providing such an example (inspired by a epidemiological case study reported in \cite{Lu2020}): due to airflow, we would expect~$\alpha_{12,6}$,~$\alpha_{12, 2}$, and $\alpha_{6,2}$ to be much larger than all other position-dependent transmission rates.  Note that due to the direction of airflow, we would also expect (significantly) asymmetric transmission rates (e.g., $\alpha_{12,2}\gg\alpha_{2,12}$).

A particularly applicable (and fairly general) special case of position-dependent transmission rates is as follows: we partition the set of possible positions into a set of \emph{zones}, $Z_1,Z_2,\ldots$.  Let $Z(m)$ denote the zone to which position $m$ belongs. Then, each zone $Z$ could have its own transmission rate $\alpha_Z$ such that $\alpha_{m,j}=\alpha_{Z(m)}$ whenever $Z(m)=Z(j)$ (i.e., positions $m$ and $j$ are in the same zone), while~$\alpha_{m,j}=0$, otherwise.  For example, we could have a retail store where some people are lined up outside the store and others are inside the store shopping or checking out.  Hence, we have two zones, and customers can only infect others in the same zone\footnote{We could of course add additional zones for e.g., the produce department, the deli, the bakery, etc.}.  Moreover, we would expect a much lower transmission rate outdoors than indoors.  Of course, we can relax the assumption that~$\alpha_{m,j}=0$ when $Z(m)\neq Z(j)$, instead opting for ``cross-zone transmission rates'' $\alpha_{m,j}=\alpha_{Z(m),Z(j)}$, where presumably $\alpha_{Z,Z'}<\min(\alpha_Z,\alpha_{Z'})$ for all pairs of zones $Z$ and $Z'$.
\begin{figure}[]
\centering
\includegraphics[width=0.50\textwidth]{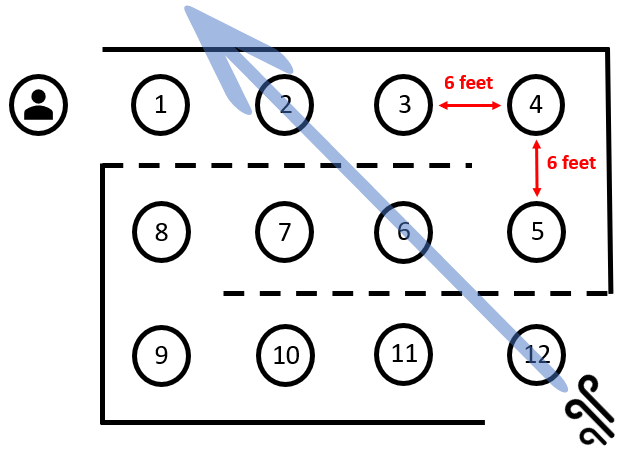}
\caption{A ``snake queue'' where 12 queue positions are organized in a $3\times4$ rectangular array such that any two orthogonally adjacent positions are 6 feet apart.  Moreover,  a strong current of air is flowing through positions 12, 6, and 2 (in that order).}
\label{fig:snake}
\end{figure}

For further discussion of such transmission models---including the derivation of~$\rnot$ for an M/M/1 system under a general position-dependent transmission model (Proposition~\ref{prop:distance}) and a closed-form result for $\rnot$ in the special case where transmissions can only occur between customers that are within a fixed number of positions (e.g., distance) from one another (Proposition~\ref{prop:queue_pos})---see Appendix~\ref{app:distance}.

\subsection{Concluding Remarks}
In this paper, we have introduced a modeling framework for studying disease transmission in service facilities where customer arrivals and departures exhibit idiosyncratic stochastic variability. In addition to proposing a transmission model (embedded in a queueing-theoretic context) and measures of disease transmission (e.g., $\rnot$, $\lambda p\rnot$, $\rnothr$, $\rnotlr$, etc.) in the context of service facilities, we have proposed the novel queueing-theoretic notion of sojourn time overlaps as a methodological tool for deriving our risk measures.  

We hope that our modeling framework will provide researchers with ample new research opportunities and some ideas that can help with studying new classes of problems at the intersection of queueing theory and epidemiology.  Specifically, on the theoretical side, we posit sojourn time overlaps (as introduced in this paper) merit further study as mathematical objects in their own right; their study can potentially lead to a greater understanding of the dynamics of a variety of fundamental queueing models. Interest in this notion has the potential to generate a rich class of new open problems in queueing.  Our preliminary investigations toward deriving the $\rnot$ metric for more sophisticated (and realistic) queueing networks suggest that the analysis of sojourn time overlaps in these systems is often nontrivial, yet still tractable.

Meanwhile, given that we have been able to derive a variety of formulas for our transmission risk measure $\rnot$ in closed form, we anticipate that our modeling framework can be incorporated in stylized economic models that allow for the extraction of additional managerial insights.  Finally, on the practical side, we are optimistic that the methods presented in this paper can supplement existing tools (e.g., simulation, spatio-temporal models informed by mobility networks, etc.) in assisting managers and policymakers alike in making informed decisions when designing and fine-tuning appropriate interventions for mitigating disease transmission in service facilities.  

\bibliographystyle{arXiv2021} 
\bibliography{reference.bib}

\newpage

\APPENDIX
\quad
\section{Details omitted from the proof of Proposition~\ref{prop:main}}\label{app:double}
In this section, we provide more formal details for completing the proof of Proposition~\ref{prop:main}.

For any arbitrary value of $t>0$, consider the number of other customers---who must all be susceptible by assumption---$\mathbf{Ov}(t)$ that the IC's sojourn overlaps with for a nonzero duration of time that is less than or equal to $t$.  If we further let $\mathbf{OvB}(t)$ and $\mathbf{OvA}(t)$ be the set of such customers that arrived \emph{before} and \emph{after} the IC, respectively, then $\mathbf{Ov}(t)=\mathbf{OvB}(t)+\mathbf{OvA}(t)$ (since we are assuming Poisson arrivals, no other customer's arrival coincided with the instant of the IC's arrival with probability 1, so each other customer arrived either before or after the IC).  Moreover, since the system is ergodic (and assumed to be in steady state) and since infectious customers are \emph{functionally indistinguishable} from their susceptible counterparts, $\mathbf{OvB}(t)$ and $\mathbf{OvA}(t)$ are distributed in the same way (although not necessarily independent); this is because given any pair of customers, where one is susceptible and the other is infected, it is equally likely that that infected one arrived before or after the other, and the same is true when conditioning on the pair's sojourn time overlap.

With this new notation and the fact that $\mathbf{OvA}(t)$ and $\mathbf{OvB}(t)$ are distributed in the same way, we can rewrite $\rnot$ as follows:
\[\rnot=\mathbb E\left[\int_0^\infty\mathbb P(t\ge\theta)\,d\mathbf{Ov}(t)\right]=\mathbb E\left[\int_0^\infty\mathbb P(t\ge\theta)\,d(\mathbf{OvA}(t)+\mathbf{OvB}(t))\right]=2\mathbb E\left[\int_0^\infty \mathbb P(t\ge\theta)\,d\mathbf{OvB}(t)\right],\] where the expectation operators are needed as for any given value of $t$, the $\mathbf{Ov}(t)$, $\mathbf{OvA}(t)$, and  $\mathbf{OvB}(t)$ are random variables.  Now observe that the last equality describes twice the expected number of customers that the IC will infect among those already present in the system, which is precisely twice the quantity given in Display~\eqref{eq:half}, completing the proof.
\section{The values of $W_i^{(s)}$ in the M/M/$c$ Model}\label{app:ws}
In this appendix, we prove that the sojourn time overlap between the IC and customer $i$ is distributed according to

\[W_i^{(s)}\sim\begin{cases}\expd(2\mu)&1\leq s<c\\\min(\erld(s-c+1,(c-1)\mu)+\expd(\mu),\expd(\mu))&i\leq c\leq s\\\erld(i-c,c\mu)+W_c^{(s-(i-c))}&c<i\leq s\end{cases},\] where all component distributions above are independent of one other, and (as claimed in Proposition~\ref{prop:mmc}) takes the following Laplace Transform:

\[\widetilde{W}_i^{(s)}(\alpha)=\begin{cases}2/(\eta+2)& 1\leq s<c\\\left(\eta \left(\dfrac{c-1}{\eta+c}\right)^{s-c+1}+\eta+2\right)\left(\dfrac{1}{\eta^2+3\eta+2}\right)& i\leq c\leq s\\[10pt]\left(\dfrac{c}{\eta+c}\right)^{i-c} \left(\eta \left(\dfrac{c-1}{\eta+c}\right)^{s-i+1}+\eta+2\right)\left(\dfrac{1}{\eta^2+3\eta+2}\right)&c<i\leq s\end{cases},\] for all~$s\in\mathcal S$ and~$i\in\{1,2,\ldots,s\}$,

We addresses case 1 (i.e., when~$1\le s<c$), case 2 (i.e., when $i\le c\le s$), and case 3 (i.e., when $c<i\le s$), separately and sequentially.

\textbf{Under case 1} (i.e., when~$1\le s<c$), the IC's service starts upon arrival, seeing some customer~$i$ who is already in service at that time. Therefore, the IC's sojourn overlaps with that of customer~$i$ for an amount of time that is whichever is less of the service time of the IC (call this~$X$) or the \emph{remaining} service time of customer~$i$ (call this~$X_i$), i.e.,~$W_i^{(s)}=\min(X,X_i)$. Clearly,~$X\sim\expd(\mu)$, but we must also have~$X_i\sim\expd(\mu)$, due to the memoryless property of the exponential distribution.  Moreover, since~$X$ and~$X_i$ are independent, we must have~$W_i^{(s)}=\min(X,X_i)\sim\expd(2\mu)$ and~$\widetilde{W}_i^{(s)}(\alpha)=2/(\eta+2)$ as claimed.

\textbf{Under case 2} (i.e., when~$i\le c\le s$), the IC arrives at position~$s-c+1$ of the \emph{queue} (i.e., so that the IC will enter service at one of the~$c$ servers after the system experiences~$s-c+1$ service departures), while customer~$i$ is already in service.  In this case, the IC's sojourn overlaps with that of customer~$i$ for an amount of time equal to whichever is less of the \emph{sojourn} time of the IC or the \emph{remaining} service time of customer~$i$ (call this~$X_i$;~$X_i\sim\expd(\mu)$ as in the previous case).  The sojourn time of the IC is~$Y+X$ (so that~$W_i^{(s)}=\min\left(Y+X,X_i\right)$), where~$Y$ and~$X$ are the durations of time the IC spends in the \emph{queue} and \emph{in service}, respectively.  The random variable~$Y$ (the distribution of which depends on~$s$) corresponds to the time it takes for~$s-c+1$ successive departures (from any of the~$c$ servers), while~$X\sim\expd(\mu)$ as in the previous case.  Note that the IC is in the queue during the entire time it takes for these~$s-c+1$ successive departures (that makes up~$Y$) to take place, and hence, all~$c$ servers are busy during this time, and so these~$s-c+1$ successive departures will each take up an amount of time that is drawn from the~$\expd(c\mu)$ distribution, and since all such ``inter-departure'' times are independent (as service times are independent and exponentially distributed), we have~$Y\sim\erld(s-c+1,c\mu)$.  Note, however, that~$Y$ and~$X_i$ are \emph{not} independent (while~$X$ is independent of both~$Y$ and~$X_i$), because a departure from the system may actually be due to customer~$i$ being served.  We can alternatively view~$W_i^{(s)}=\min\left(Y'+X,X_i\right)$, where~$Y'\sim\erld(s-c+1,(c-1)\mu)$ is the time it takes for~$s-c+1$ successive departures to occur assuming only~$c-1$ servers are running (i.e., ignoring the server on which job~$i$ is running); in this case~$Y'$ and~$X_i$ are independent, and we have~$W_i^{(s)}\sim\min(\erld(s-c+1,(c-1)\mu)+\expd(\mu),\expd(\mu))$, as claimed.

We obtain~$\widetilde W_i^{(s)}$ using first-step analysis by observing that~$W_i^{(s)}$ is in fact distributed according to a Coxian phase-type distribution \citep[See][Chapter 21.1, for details]{harchol2013performance} with~$s-c+2$ phases, all but the last of which have a rate of~$c\mu$ as they correspond to a service at any of the~$c$ servers; at the conclusion of each of these phases the entire process terminates with probability~$1/c$ (corresponding to customer~$i$'s service, as this would conclude the sojourn overlap) or continue to the next phase (correspond to a departure due to any customer in service other than customer~$i$, as this would advance the IC one position in the queue, or bring them into service if they previously at the head of the queue). The last phase has a rate of~$2\mu$ as it corresponds only to the service of either the IC or server~$i$ (either of which would conclude the sojourn overlap).  Denote by~$U_m$ the remaining duration of such a distribution given that we currently have~$m$ phases left to go after the current phase (assuming the process does not terminate early), so that~$U_m=X_m+(c-1)U_{m-1}/c$, for all~$m\ge1$ where~$U_m\sim\expd(c\mu)$, while~$U_0\sim\expd(2\mu)$.  Clearly,~$W_i^{(s)}=U_{s-c+1}$.  Using standard manipulations of Laplace Transforms \citep[See][Chapter 25, for details]{harchol2013performance} and recalling that~$\eta\equiv\alpha/\mu$, we have
\[\widetilde U_m(\alpha)=\left(\frac{1}{\eta+c}\right)\left(1+(c-1)\widetilde U_{m-1}(\alpha)\right)\] for all~$m\ge1$, and~$\widetilde U_0(\alpha)=2\mu/(\alpha+2\mu)=2/(\eta+2)$.  Solving this linear recursion (and recalling that~$\eta\equiv\alpha/\mu$) yields \[\widetilde U_m(\alpha)=\left(\eta\left(\frac{c-1}{\eta+c}\right)^m+\eta+1\right)\left(\frac1{\eta^2+3\eta+2}\right),\] which coincides with the claimed value of~$\tilde W_i^{(s)}(\alpha)$ for case 2, when we set~$m=s-c+1$, thus verifying the claim.

\textbf{Under case 3} (i.e., when~$c<i\le s$), the IC arrives at position~$s-c+1$ of the queue, and finds customer~$i$ in position~$i-c$ of the queue.  We break up the sojourn time overlap between the IC and customer~$i$ into two parts: the duration of time their sojourns overlap while both the IC and customer~$i$ are present in the \emph{queue} (call this~$Q$, which depends on~$i$), and the remaining portion of the sojourn time overlap, which corresponds to the duration of time their sojourns overlaps while customer~$i$ is in service (call this~$V$, and note that the IC may---but need not necessarily be---in service during some of this time).  The first of these durations corresponds to the time it takes for the system to experience~$i-c$ consecutive departures, so~$Q\sim \erld(i-c,c\mu)$.  Meanwhile, when customer~$i$ enters service, the IC will be in position~$s-c+1-(i-c)=s-i+1$ of the queue, and hence, the remaining sojourn time overlap,~$V$, will be the same as the total sojourn time overlap between a customer who had arrived to position~$s-i+1$ of the queue (i.e., who had arrived to a system with~$s-(i-c)$ \emph{other} customers already present in the system), while customer~$i$ was in service (e.g., in position~$c$).  That is,~$V\sim W_c^{(s-(i-c))}$; also note that~$Q$ and~$V$ are independent.  It follows that in this case~$W_i^{(s)}=Q+V\sim\erld(i-c,c\mu)+W_c^{(s-(i-c))}$ as claimed.  Moreover, it follows that ~$\widetilde W_i^{(s)}(\alpha)=\widetilde Q(\alpha)\widetilde V(\alpha)=(c/(\eta+c))^{i-c}\widetilde W_{c}^{(s-(i-c))}(\alpha)$.  Substituting in the expression for~$\widetilde W_{c}^{(s-(i-c))}$ from case 2 shows that~$\widetilde W_i^{(s)}(\alpha)$ is also as claimed in case 3.
\section{Expressions for the values appearing in Proposition~\ref{prop:priority1}}\label{app:priority}

In the setting considered in Proposition~\ref{prop:priority1}, let $\pi(h,\ell)$ be the limiting probability distribution of the number of high- and low-risk customers in the system under steady state (see \cite{marks1973priority} for the exact solutions), the expressions for the five terms are given in the following Proposition:
\begin{Proposition}\label{prop:apppriority}
In the M/M/1 system with preemptive priorities described above, we have
\begin{align}
  \rhhb&=\left(\frac{\rho_\hsf}{1-\rho_\hsf}\right)\left(\frac{\eta}{\eta+1-\rho_\hsf}\right)\label{eqapp1}\\
  \rhlb&=\sum_{h=0}^{\infty}{\sum_{\ell=1}^{\infty}{\psys(h,\ell)\sum_{i=1}^{\ell}{\left(1-{\left(\dfrac{1}{1+\eta}\right)}^{h+1}\right)}}}\label{eqapp2}\\
  \rlhb&=\left(\frac{\rho_\hsf}{1-\rho_\hsf}\right)\left(\frac{\eta}{\eta+1-\rho_\hsf}\right)\label{eqapp3}\\
  \rllb&=\sum_{h=0}^{\infty}{\sum_{\ell=1}^{\infty}{\psys(h,\ell)\sum_{i=1}^{\ell}{\mathbb{P}(\wll\ge\theta)}}}\label{eqapp4}\\
  \rlla&=\sum_{(h,\ell)\in\mathcal{S}}{\psys(h,\ell)A}\label{eqapp5}
\end{align}
where \[A=\dfrac{1-\rho_\hsf-(1+\eta)^{h}\left(1-\rho_\hsf-\eta\left(h+\eta+h\eta+\ell\eta-h\rho_\hsf\right)\right)}{\eta(1-\rho)(1+\eta-\rho)(1+\eta)^h}-(1+\ell)\left(1-\dfrac{1}{1+\eta}\right)\] and $\wll$ denotes the length of a busy period started by $(h+i)/\mu$ amount of work in an M/M/1 system with arrival rate $\lhr$ and service rate $\mu$.
\end{Proposition}
\begin{proof}{Proof of Proposition~\ref{prop:apppriority}.}
Eq.~(\ref{eqapp1}) and Eq.~(\ref{eqapp3}) follow from Proposition~\ref{prop:mm1} and the observation that in both cases, we only consider the SCs who are all high-risk customers, so we can treat it as an M/M/1/FCFS system which only has high-risk customers (load becomes $\rho_\hsf$). We can get Eq.~(\ref{eqapp2}) and Eq.~(\ref{eqapp4}) by directly applying Proposition~\ref{prop:main}, the sojourn time overlap distribution in the case of Eq.~(\ref{eqapp2}) follows $\erld(h+1,\mu)$ while the sojourn time overlap $\wll$ in the case of Eq.~(\ref{eqapp4}) not only depends on the present number of customers (both high- and low-risk customers) but also will be affected by the future arrival of high-risk customers (we defer the derivation of $\mathbb{P}(\wll\ge\theta)$ right after this proof). We finish this proof by finding the expression of $\rlla$, when the low-risk IC arrives.   Assuming the system state is $(h,\ell)$, there will be $\ell$ low-risk SCs and $h$ high-risk SCs in the system, all of whom will leave the system before the IC leaves.   Note that more high-risk customers may arrive before the IC leaves, and they will be served before the IC.  Therefore, we define the first busy period (with $h$ high-risk SCs) as the time until the first low-risk customer leaves the system.  Each remaining busy period will end when the next low-risk customer leaves the system. Since only future high-risk SCs will affect the process, we define $V(x,y)$ as the expected number of arrivals to an M/M/1 system with $\rho=\rho_H$ in state $y$ during current busy period before the next service completion of a low-risk customer given the initial state $x$. $V(x,y)$ can be solved by the following system:
\[\begin{cases}&V(x,y)=\dfrac{\rho_H}{1+\rho_H}V(x+1,y)+\dfrac{1}{1+\rho_H}V(x-1,y),\quad\forall~1\leq x\leq y\\&V(1,y)=\dfrac{\rho_H}{1+\rho_H}V(2,y),\\&V(y,y)=\dfrac{\rho_H}{1+\rho_H}(1+V(y,y))+\dfrac{1}{1+\rho_H}V(y-1,y)\end{cases}\] which leads to
\begin{align}
    V(x,y)=\sum_{j=(y+1-x)^{+}}^{y}{{(\rho_H)}^j}.\label{eqapp6}
\end{align}
Hence, the first busy period will contribute $\sum_{n=1}^{\infty}{V(h+1,n)(1-(1/1+\eta)^{n+1})}$ to our risk metric while the other $\ell$ busy periods will contribute $\ell\sum_{n=1}^{\infty}{V(1,n)(1-(1/1+\eta)^{n+1})}$, together with Proposition~\ref{prop:main} we get 
\begin{align}
   \rlha=\sum_{(h,\ell)\in\mathcal{S}}{\psys(h,\ell)\left(\sum_{n=1}^{\infty}{[V(h+1,n)+\ell V(1,n)] \left(1-(\dfrac{1}{1+\eta})^{n+1}\right)}\right)}.\label{eqapp7} 
\end{align}
Substituting Eq~(\ref{eqapp6}) into Eq.~(\ref{eqapp7}) and simplifying the formulas yield the claimed result in Eq.~(\ref{eqapp5}). \hfill$\square$
\end{proof}
Next we proceed to derive the expression of  $\mathbb{P}(\wll\ge\theta)$. According to the definition of $\wll$, we have \citep[See chapter 27 of][for details on the Laplace transform of the busy period]{harchol2013performance} \[\mathbb{P}(\wll\ge\theta)=1-\widetilde{W}_{\lsf\to\lsf(i)}^{(h,\ell)}(\alpha+\lhr-\lhr\widetilde{B}(\alpha))\] where \[\widetilde{W}_{\lsf\to\lsf(i)}^{(h,\ell)}(\mathbf{s})=(\dfrac{\mu}{\mu+\mathbf{s}})^{h+i},\] and  \[\widetilde{B}(\alpha)=\dfrac{1}{2\lhr}\left(\lhr+\mu+\alpha-\sqrt{{(\lhr+\mu+\alpha)}^{2}-4\lhr\mu}\right).\]
\section{Hyperexponential Transmission Thresholds}\label{hyper}

When transmission thresholds are hyperexponentially distributed, we have the following decomposition result, in terms of systems with exponentially distributed transmission thresholds.  \begin{Proposition}\label{prop:hyper}
If transmission thresholds are hyperexponentially distributed so that there exists some set of infectious-susceptible customer pair types~$\mathcal J$ such that~$\theta\sim\expd(\alpha_j)$ with probability~$q_j$ and $\sum_{j\in\mathcal J}q_j=1$, then
\[\rnot=\sum_{j\in\mathcal J}q_j\rnot[j],\]
where~$\rnot[j]$ is the~$\rnot$ when~$\theta\sim\expd(\alpha_j)$ for all infectious-susceptible customer pairs.
\end{Proposition}
\begin{proof}{Proof}
First observe that in this setting we have \[\mathbb P\left(W_i^{(s)}\ge\theta\right)=\sum_{j\in\mathcal J}q_j\mathbb P\left(\left.W_i^{(s)}\ge\theta\right|\theta\sim\expd(\alpha_j)\right)=\sum_{j\in\mathcal J}q_j\left(1-\widetilde W_i^{(s)}(\alpha_j)\right),\] and so, following Eqs.~\eqref{eqn:r0} and \eqref{eqn:r0exp}, the claim follows:
\[\rnot=2\sum_{s\in\mathcal S}\pi(s)\sum_{i=1}^{n(s)}\left(\sum_{j\in\mathcal J}q_j\left(1-\widetilde W_i^{(s)}(\alpha_j)\right)\right)=\sum_{j\in\mathcal J}q_j\left(2\sum_{s\in\mathcal S}\pi(s)\sum_{i=1}^{n(s)}\left(1-\widetilde W_i^{(s)}(\alpha_j)\right)\right)=\sum_{j\in\mathcal J}q_j\rnot[j]. \]
\hfill$\square$
\end{proof}

\section{A Supplemental Discussion on the Impact of Spatial Positioning on Transmission}
Consider the setting where transmission thresholds are position-dependent as discussed in Section~\ref{sec:extension-distance}.  In this setting, the following proposition gives $\rnot$ for an~M/M/1/FCFS system:
\begin{Proposition}\label{prop:distance}
Consider an~M/M/1/FCFS system where transmission rates~$\alpha_{m,j}$ depend on positions as described in Section~\ref{sec:extension-distance}. Letting~$\eta_{m,j}\equiv\alpha_{m,j}/\mu$, we have
\[\rnot=\left(\dfrac{2\rho}{1-\rho}-(1-\rho)\sum_{s=0}^\infty\rho^s\sum_{i=1}^s\left(2-\prod_{j=0}^{i-1}\left\{\dfrac{1}{(\eta_{s+1-j,i-j})+1}\right\}-\prod_{j=0}^{i-1}\left\{\dfrac{1}{(\eta_{i-j,s+1-j})+1}\right\}\right)\right).\]
\end{Proposition}
\begin{proof}{Proof of Proposition~\ref{prop:distance}.}
This proof follows a similar argument to that presented in  Proposition~\ref{prop:mm1}.  The first crucial difference is that the probability that the IC infects the SC initially at position $i$ is not $1-(\eta+1)^{-i}$ in this setting, but rather, it is and is given by $1-\prod_{j=0}^{i-1}\left\{\frac{1}{(\eta_{s+1-j,i-j})+1}\right\}$.  This is because initially the IC is in position $s+1$ while the SC is in position $i$, then the IC is in position $s$ while the SC is position $i-1$, and so on, until the SC is in position $s+1-i$, while the IC is in position $1$.  The IC and SC (that was initially at position $i$) spend a duration of time that is distributed $\expd(\mu)$ in each of these $i$ positional configuration.  Hence, during the IC's sojourn in position $s+1-j$, given that the SC did not previously become infected, the IC \emph{fails} to infect the SC (who is concurrently in position $i-j$), with probability $\mu/((\alpha_{s+1-j,i-j})+\mu)=1/((\eta_{s+1-j,i-j})+1)$, from which the claimed infection probability follows.  Following the proof of Proposition~\ref{prop:mm1}, the expected number of customers that the IC infects among those who arrived \emph{before} the IC is given by \begin{align}\label{eqn:pos}\dfrac{\rho}{1-\rho}-(1-\rho)\sum_{s=0}^\infty\rho^s\sum_{i=1}^s\left(1-\prod_{j=0}^{i-1}\left\{\dfrac{1}{(\eta_{s+1-j,i-j})+1}\right\}\right).\end{align} Given that the M/M/1 system is time-reversible \cite[See chapters 9 and 13 of][for details]{harchol2013performance}, together with the assumption that the IC is functionally indistinguishable from SCs, and the fact that we are considering a first-come-first-serve system, the expected number of customers that the IC infects among those who arrived \emph{after} the IC is given by a modified version of the same formula given in Display~\eqref{eqn:pos}: the only modification is that $\eta_{s+1-j,i-j}$ is replaced by $\eta_{i-j,s+1-j}$ (i.e., we have reversed indices).  This modification is due to the fact that the symmetry introduced by time-reversibility does not necessarily apply to the infection rates between pairs of positions (i.e., $\alpha_{i,j}$ need not be equal to $\alpha_{j,i}$, and hence, $\eta_{i,j}$ need not be equal to $\eta_{j,i}$).  The claimed result then follows by summing these two expectations. \hfill $\square$
\end{proof}

We  proceed to  discuss  a  special  case of position-dependent  transmission  rates where rates depend on distance.  Assuming a queue proceeding in a straight line where distances between successive customers are the same, we consider a transmission model where $\alpha_{i,j}=\alpha I\{|i-j|\le d\}$, where $I\{\cdot\}$ denotes the indicator function.  That is, an IC can only infect those customers who are waiting up to $d$ positions in front of or behind them in the queue.  This model would be reasonable, if, e.g., successive customers in the queue are spaced exactly 6 feet apart and we believe that there is a non-negligible transmission risk (occurring with rate $\alpha$) when customers are spaced 6--18 feet, but the risk is assumed to be negligible when customers are spaced 24 or more feet apart; in this example, $d=3$.  In the case of an M/M/1/FCFS queue, we can compute the $\rnot$ value for this distance-based transmission model in closed form:
\begin{Proposition}\label{prop:queue_pos}
Consider the same M/M/1/FCFS system as in Proposition \ref{prop:mm1} where a susceptible customer can only be infected by an infected customer within $d$ positions in the queue from themselves.  Then we have
\small{
\[ \rnot = \frac{2\rho\left( ((1+\eta)\rho)^d(2\rho-1)+\eta^2(1+\eta)^d(\rho^2-1)+\rho^d\left( (1-\rho)^2 - (1+\eta)^d \left( \left( \frac{1}{1+\eta} \right)^d(1-\rho)^2 + 2\rho - 1 \right) \right) \right)}{\eta(1+\eta-\rho)(\rho-1)(1+\eta)^d} \]}
\end{Proposition}
\begin{proof}{Proof of Proposition~\ref{prop:queue_pos}.}
Noting that a tagged IC overlaps with, at most, $d$ customers at the back of the queue upon their arrival, we can again follow the same symmetry argument as in Proposition \ref{prop:mm1} and condition on whether or not the length of the queue on arrival exceeds the threshold amount $d$ to find
\small{
\begin{align*}
    \rnot &= 2\sum_{s=0}^\infty \pi(s) \sum_{k=1}^{\min\{s,d\}} \left( 1 - \widetilde{W}_{s-k+1}^{(s)}(\alpha) \right) \\
    &= 2\left( \sum_{s=0}^d \pi(s)\sum_{k=1}^s \left( 1 - \widetilde{W}_{s-k+1}^{(s)}(\alpha) \right) + \sum_{s=d+1}^\infty \pi(s) \sum_{k=1}^d\left( 1 - \widetilde{W}_{s-k+1}^{(s)}(\alpha)\right) \right)\\
    &= 2\left( \sum_{s=0}^d (1-\rho)\rho^s \sum_{k=1}^s \left( 1 - \left(\frac{1}{1+\eta}\right)^{s-k+1}\right) + \sum_{s=d+1}^\infty (1-\rho)\rho^s  \sum_{k=1}^d\left( 1 - \left(\frac{1}{1+\eta}\right)^{s-k+1}\right) \right)\\
    &= \frac{2\rho\left( ((1+\eta)\rho)^d(2\rho-1)+\eta^2(1+\eta)^d(\rho^2-1)+\rho^d\left( (1-\rho)^2 - (1+\eta)^d \left( \left( \frac{1}{1+\eta} \right)^d(1-\rho)^2 + 2\rho - 1 \right) \right) \right)}{\eta(1+\eta-\rho)(\rho-1)(1+\eta)^d}.
\end{align*}}
\hfill$\square$
\end{proof} \label{app:distance}

\end{document}